\documentclass[a4paper]{amsart}

\usepackage{hyperref}
\usepackage{txfonts, amsmath,amstext,amsthm,amscd,amsopn,verbatim,amssymb, amsfonts, mathabx}

\usepackage{xcolor}
\usepackage{diagbox}
\usepackage{array}
\newcommand{\PreserveBackslash}[1]{\let\temp=\\#1\let\\=\temp}
\newcolumntype{C}[1]{>{\PreserveBackslash\centering}p{#1}}
\newcolumntype{R}[1]{>{\PreserveBackslash\raggedleft}p{#1}}
\newcolumntype{L}[1]{>{\PreserveBackslash\raggedright}p{#1}}

\usepackage{subdepth}
\usepackage{tikz}
\usetikzlibrary{math}
\usetikzlibrary{positioning}

\usepackage{fullpage}
\usepackage{todonotes}
\usepackage{multicol}

\usepackage[bbgreekl]{mathbbol}
\usepackage{enumerate}

\usepackage{float}
\restylefloat{table}

\usepackage{color, colortbl}
\usepackage{array}
\usepackage{varwidth} 

\definecolor{Gray}{gray}{0.9}

\newcolumntype{g}{>{\columncolor{Gray}}c}
\newcolumntype{M}{V{4cm}} 

\usepackage{tikz}
\usepackage{tikz-cd}
\usetikzlibrary{matrix}
\usetikzlibrary{shapes}
\usetikzlibrary{arrows, automata}
\usetikzlibrary{calc,3d}
\usetikzlibrary{decorations,decorations.pathmorphing,decorations.pathreplacing,decorations.markings}
\usetikzlibrary{through}

\tikzset{snakeit/.style={decorate, decoration={snake, amplitude=.2mm,segment length=1mm}}}

\tikzset{ext/.style={circle, draw,inner sep=1pt}, int/.style={circle,draw,fill,inner sep=2pt},nil/.style={inner sep=1pt}}
\tikzset{cy/.style={circle,draw,fill,inner sep=2pt},scy/.style={circle,draw,inner sep=2pt},scyx/.style={draw,cross out,inner sep=2pt},scyt/.style={draw,regular polygon,regular polygon sides=3,inner sep=0.95pt}}
\tikzset{exte/.style={circle, draw,inner sep=3pt},inte/.style={circle,draw,fill,inner sep=3pt}}
\tikzset{diagram/.style={matrix of math nodes, row sep=3em, column sep=2.5em, text height=1.5ex, text depth=0.25ex}}
\tikzset{diagram2/.style={matrix of math nodes, row sep=0.5em, column sep=0.5em, text height=1.5ex, text depth=0.25ex}}
\tikzset{rowcolsep/.style={column sep=.2cm, row sep=.1cm}}

\tikzset{
  crossed/.style={
    decoration={markings,mark=at position .5 with {\arrow{|}}},
    postaction={decorate},
    shorten >=0.4pt}}

\tikzset{every picture/.style={baseline=-.65ex} }
\tikzset{every loop/.style={draw}}

  \tikzset{->-/.style={decoration={
    markings,
    mark=at position .5 with {\arrow{>}}},postaction={decorate}}}

\theoremstyle{plain}
  \newtheorem{thm}{Theorem}
  
  \newtheorem{prop}[thm]{Proposition}
  
  \newtheorem{cor}[thm]{Corollary}
  \newtheorem{conjecture}[thm]{Conjecture}
  \newtheorem{lemma}[thm]{Lemma}
\theoremstyle{definition}

\numberwithin{thm}{section}

\newcommand{\p}{\partial}

\newcommand{\Z}{{\mathbb{Z}}}

\newcommand{\Graphs}{{\mathsf{Graphs}}}

\newcommand{\FreeLie}{\mathrm{FreeLie}}

\newcommand{\op}{\mathcal}

\newcommand{\Lie}{\mathsf{Lie}}

\newcommand{\Com}{\mathsf{Com}}

\newcommand{\bpm}{\begin{pmatrix}}
\newcommand{\epm}{\end{pmatrix}}

\newcommand{\bbS}{\mathbb{S}}

\DeclareMathOperator{\Aut}{Aut}

\DeclareMathOperator{\sgn}{sgn}
\DeclareMathOperator{\gr}{gr}

\newcommand{\MM}{{\mathcal M}}

\newcommand{\Q}{\mathbb{Q}}

\newcommand{\STr}{\mathrm{STr}}

\newcommand{\COp}{{\op C}}
\newcommand{\DOp}{{\op D}}

\newcommand{\BOp}{{\op B}}
\newcommand{\AOp}{{\op A}}

\newcommand{\Ind}{{\mathrm{Ind}}}

\newcommand{\Seq}{{\mathsf{Seq}}}
\newcommand{\opOne}{{\mathbf{1}}}

\newcommand{\Pois}{{\mathsf{Pois}}}
\newcommand{\BVGraphs}{{\mathsf{BVGraphs}}}

\newcommand{\BV}{{\mathsf{BV}}}

\newcommand{\fHHGC}{\ftX{}}
\newcommand{\DS}{\Delta_0}
\newcommand{\cro}{\mathrm{cr}}

\newcommand{\tX}{\check X}
\newcommand{\ftX}{\widecheck{fX}}

\newcommand{\tG}{\check G}
\newcommand{\ftG}{\widecheck {fG}}
\newcommand{\rd}{{\mathrm{red}}}
\newcommand{\croBV}{\BV^\rd}
\newcommand{\croBVGraphs}{\BVGraphs^\rd}
\newcommand{\croGraphs}{\Graphs^\rd}

\author{Sam Payne}
\address{Department of Mathematics \\ 2515 Speedway, PMA 8.100 \\ Austin, TX 78722 USA}
\email{sampayne@utexas.edu}
\author{Thomas Willwacher}
\address{Department of Mathematics\\ ETH Zurich\\ R\"amistrasse 101 \\ 8092 Zurich, Switzerland}
\email{thomas.willwacher@math.ethz.ch}

\begin{document}
\title{The weight two compactly supported Euler characteristic of moduli spaces of curves}

\begin{abstract}
  We derive a formula for the generating function for the weight two compactly supported $\bbS_n$-equivariant Euler characteristics of the moduli spaces of curves $\MM_{g,n}$, using graph complexes and calculations inspired by operadic methods. 
\end{abstract}

 \thanks{
   S.P. has been supported by NSF grants DMS--2001502 and DMS--2053261. Portions of this work were carried out while SP was in residence at ICERM for the special thematic semester program on Combinatorial Algebraic Geometry. T.W. has been supported by the ERC starting grant 678156 GRAPHCPX, and the NCCR SwissMAP, funded by the Swiss National Science Foundation.}

\maketitle

\setcounter{tocdepth}{1}
   \tableofcontents

\section{Introduction}

The rational cohomology of moduli spaces of curves is of central interest not only in algebraic geometry but also in  topology, geometric group theory, and mathematical physics. In the stable range, the cohomology is completely understood, thanks to the celebrated proof of Mumford's conjecture via stable homotopy theory  \cite{MadsenWeiss07}.  Outside of the stable range, even the most basic properties, such as the parameters $k$, $g$, $n$ for which the rational cohomology $H^k(\MM_{g,n})$ is zero or non-zero, have remained largely mysterious \cite[\S 9]{Margalit19}.  

Algebraic geometry techniques have recently yielded new nonvanishing results and lower bounds on dimensions of cohomology groups in the unstable range \cite{CGP1, CGP2, PayneWillwacher}. The algebraic structure of $\MM_{g,n}$ as a Deligne-Mumford stack endows its rational cohomology with a mixed Hodge structure.  Nonvanishing results and lower bounds on dimensions of cohomology groups are then obtained by studying the associated graded of the weight filtration, one weight at a time. The first two nontrivial graded pieces of the weight filtration may be expressed in terms of the cohomology of commutative graph complexes including, in weight zero, the commutative graph complex originally studied by Kontsevich \cite{K3}  along with, in weight two, more general graph complexes closely related to those appearing in recent work on the embedding calculus \cite{FTW2}.

The cohomology groups of the graph complexes that arise in this context are far from fully understood.  In such cases where cohomology groups cannot be computed directly, Euler characteristics often yield valuable insights. For instance, Harer and Zagier famously computed the Euler characteristics of $\MM_g$ and deduced that these moduli spaces have cohomology groups that grow superexponentially in both even and odd degrees \cite{HarerZagier86}.  This is especially striking since the stable cohomology grows subexponentially and, at the time of their writing and for nearly twenty years after, there were no known examples where $H^k(\MM_g)$ is nonzero for odd $k$. The first such example was $H^5(\MM_4)$ \cite{Tommasi05}.  All other known examples are proved using the associated graded of the weight filtration and graph complex techniques; see \cite{CGP1, PayneWillwacher}.

Even if one is primarily interested in the cohomology of $\MM_g$, for inductive arguments using the boundary structure of the moduli space of stable curves as in \cite{ArbarelloCornalba98, GK}, it is essential to understand the cohomology of moduli spaces of curves with marked points $\MM_{g,n}$, and to understand these not only as vector spaces but as representations of the symmetric group $\bbS_{n}$, with the action induced by permuting marked points.  Thus, one is naturally led to investigate the $\bbS_{n}$-equivariant Euler characteristic of the associated graded of the weight filtration on the rational cohomology of $\MM_{g,n}$.  By Poincar\'e duality \cite[Theorem~6.3]{PetersSteenbrink}, it is equivalent to study the associated graded of the weight filtration on compactly supported cohomology.  

A closed formula for the weight 0 compactly supported Euler characteristic of moduli spaces of curves is given in \cite{CFGP}, and the weight 1 compactly supported cohomology vanishes.  Our main result is a formula for the generating function for the $\bbS_n$-equivariant Euler characteristic of the weight 2 compactly supported cohomology of $\MM_{g,n}$.

\subsection{Statement of main result} Let $V_\lambda$ denote the irreducible representation of the symmetric group corresponding to a Young diagram $\lambda$. For any graded $\Q$-vector space $W=\oplus_k W^k$ with finite dimensional graded pieces and an action of the symmetric group $\bbS_n$, we consider the decomposition into isotypical components
\[
  W^k \cong \bigoplus_\lambda V_\lambda \otimes \Q^{n_{\lambda,k}}.
\]
Then the equivariant Euler characteristic of $W$ is defined to be the symmetric function
\[
\chi^\bbS(W)
=
\sum_{k,\lambda} (-1)^k n_{\lambda,k} s_\lambda,
\]
with $s_\lambda$ the Schur function corresponding to $\lambda$.

We are particularly interested in the associated graded of the weight filtration 
$$
\gr_i H^\bullet_c (\MM_{g,n}) := W_i H_c^\bullet(\MM_{g,n}) / W_{i-1} H^\bullet_c(\MM_{g,n}),
$$
which carries a natural $\bbS_n$-action induced by permuting the marked points.  Let
$$
\chi_{i}^{\bbS}(\MM_{g,n}):= \chi^{\bbS}(\gr_{i} H_c^\bullet(\MM_{g,n})).
$$
We will give a formula for the generating function for the weight $2$ equivariant Euler characteristics
\begin{equation}\label{equ:omega2def}
\omega_2 := 
\sum_{g,n\atop 2g+n\geq 3} 
    \chi^{\bbS}_2 (\MM_{g,n})\hbar^g
\end{equation} 

\noindent in the ring of formal symmetric power series $\Lambda=\Q\llbracket p_1,p_2,\dots \rrbracket$.  Here $p_\ell = x_1^\ell + x_2^\ell + x_3^\ell + \cdots$ is the degree $\ell$ power sum.  Let  $P_\ell:=1+p_\ell$ be the inhomogeneous power sum.  We define 
\begin{align}\label{equ:Zelldef}
  Z_\ell &:= \frac 1 \ell \sum_{d\mid \ell} \mu(\ell/d) \frac {P_d } {\hbar^d},
\end{align}
where $\mu$ is the M\"obius function and $\hbar$ is a formal variable that will count the genus, as well as 
\begin{equation}
  \label{equ:psi def}
  \psi_0(z) := -\sum_{j=1}^\infty \frac{B_j}{j} \frac 1 {(-z)^j} \mbox{ \ \ \ and \ \ \ }
  \psi_1(z) := -\sum_{j=0}^\infty B_j \frac 1 {(-z)^{j+1}}.
\end{equation}
Note that $\psi_0(z)+\log z$ and $\psi_1(z)$ are the asymptotic expansions as $z\to \infty$ of the digamma and trigamma function respectively.
We identify the logarithm with its power series expansion 
\[
  \log(1-z) = -\sum_{j=1}^\infty z^j.
\]

\begin{thm}\label{thm:eulerchar}
The generating function for $\bbS$-equivariant weight $2$ compactly supported Euler characteristics of moduli spaces of curves is 
\begin{equation} \label{equ:eulerchar}
\resizebox{.95\hsize}{!}{
$\begin{aligned}
    \omega_2 = \frac{(\hbar-1) P_1 }{2} & \left[
      \left(  -\frac \hbar {P_1} + \sum_{\ell\geq 1} \frac{\mu(\ell)}{\ell}\Big(\log(\ell \hbar^\ell Z_\ell) +\psi_0(-Z_\ell)\Big) \right)^2
      + \sum_{\ell\geq 1}  \frac{\mu(\ell)}{\ell} \Big( \log(2\ell\hbar^{2\ell}Z_{2\ell}) +\psi_0(-Z_{2\ell}) \Big) 
  +\sum_{\ell\geq 1} \left(\frac{\mu(\ell)}{\ell}\right)^2 \psi_1(-Z_\ell)
  -\frac{\hbar^2}{P_1^2}
    \right] + \cdots  \\ 
  & \cdots 
    -\hbar +(\hbar^2-1) P_1 + \frac 1 2 \left(P_1^2+P_2\right).
    \end{aligned}$
    }
\end{equation}
\end{thm}

Let us discuss the convergence and structure of the formula.
First note that 
\begin{equation}\label{equ:Xell}
\frac 1 {Z_\ell} = \frac{\ell \hbar^\ell}{P_\ell(1+X_\ell)} =
\frac{\ell \hbar^\ell}{P_\ell}\sum_{j\geq 0}(-X_\ell)^j
=O(\hbar^\ell)
\quad\quad
\text{with }
\quad \quad
X_\ell := \sum_{d\mid \ell \atop d\neq \ell} \mu(\ell/d) \frac {\hbar^{\ell-d} P_d } {P_\ell} = O(\hbar^{\frac \ell2})\, .
\end{equation}
We can hence expand the terms within the square brackets in \eqref{equ:eulerchar} as power series in $\hbar$.
First,
\begin{align} \label{equ:psi0expanded}
  \log(\ell \hbar^\ell Z_\ell) +\psi_0(-Z_\ell)
  &=
  \log P_\ell 
  -
  \sum_{j\geq 1} \frac{ (-1)^j}{j} X_\ell^j
  -
  \sum_{k=1}^\infty\sum_{j=0}^\infty 
  \frac{B_k}{k} 
  { -k \choose j }
  \left(\frac{\ell \hbar^\ell}{P_\ell}\right)^k
  X_\ell^j
  \\ 
  \label{equ:psi1expanded}
  \psi_1(-Z_\ell) 
  &= 
  -\sum_{k=1}^\infty \sum_{j=0}^\infty B_{k-1} 
  { -k \choose j }
  \left(\frac{\ell \hbar^\ell}{P_\ell}\right)^{k}
  X_\ell^j.
\end{align}

Note in particular that, with the exception of the term $\log P_\ell$, the expressions \eqref{equ:psi0expanded} and \eqref{equ:psi1expanded} are (infinite) linear combinations of Laurent monomials of the form
\begin{equation}\label{equ:P monomial}
\hbar^a \frac{P_{d_1}^{b_1}P_{d_2}^{b_2}\cdots P_{d_s}^{b_s}}{P_\ell^c}  \mbox{ \ \ such that \ \ } \left\{\begin{array}{l} d_i | \ell,  \mbox{ for } 1 \leq i \leq s, \\ b_1 + \cdots + b_s \leq c, \\ a+ b_1d_1 + \cdots + b_sd_s = \ell c. \end{array} \right.
\end{equation}
In particular, if we say that $P_d$ has degree $d$, then the coefficient of $\hbar^a$ for $a>0$ in each of the expressions \eqref{equ:psi0expanded} and \eqref{equ:psi1expanded} is a finite linear combination of Laurent monomials of degree $-a$ in the inhomogeneous power sum symmetric functions, each of the special form prescribed by  \eqref{equ:P monomial}.  
Inserting \eqref{equ:psi0expanded} and \eqref{equ:psi1expanded} into \eqref{equ:eulerchar} we find 
\begin{equation} \label{equ:eulerchar expanded}
  \resizebox{.95\hsize}{!}{
  $\begin{aligned}
      \omega_2 = \frac{(\hbar-1) P_1 }{2} & \left[
        \left(  -\frac \hbar {P_1} + 
        \sum_{\ell\geq 1} \frac{\mu(\ell)}{\ell} \log P_\ell 
        +
        \sum_{\ell\geq 1} \frac{\mu(\ell)}{\ell}
        \left(
  -
  \sum_{j\geq 1} \frac{ (-1)^j}{j} X_\ell^j
  -
  \sum_{k=1}^\infty\sum_{j=0}^\infty 
  \frac{B_k}{k} 
  { -k \choose j }
  \left(\frac{\ell \hbar^\ell}{P_\ell}\right)^k
  X_\ell^j
        \right) \right)^2
\right. 
  \\ &
\left.
  + \sum_{\ell\geq 1}  \frac{\mu(\ell)}{\ell} \log P_{2\ell}
  + 
  \sum_{\ell\geq 1}  \frac{\mu(\ell)}{\ell}
  \left( 
    -
  \sum_{j\geq 1} \frac{ (-1)^j}{j} X_{2\ell}^j
  -
  \sum_{k=1}^\infty\sum_{j=0}^\infty 
  \frac{B_k}{k} 
  { -k \choose j }
  \left(\frac{2\ell \hbar^{2\ell}}{P_{2\ell}}\right)^k
  X_{2\ell}^j
    \right)
\right.
\\&
\left.
    -
    \sum_{\ell\geq 1} \left(\frac{\mu(\ell)}{\ell}\right)^2 
    \sum_{k=1}^\infty \sum_{j=0}^\infty B_{k-1} 
  { -k \choose j }
  \left(\frac{\ell \hbar^\ell}{P_\ell}\right)^{k}
  X_\ell^j
       -\frac{\hbar^2}{P_1^2} \right] 
      -\hbar +(\hbar^2-1) P_1 + \frac 1 2 \left(P_1^2+P_2\right).
      \end{aligned}$
      }
  \end{equation}
By expanding the square and collecting terms, we arrive at the following equivalent version of Theorem \ref{thm:eulerchar}.

\begin{cor}\label{cor:eulerchar}
    Let $A_g$ and $C_g$ be the coefficients of $\hbar^g$ in the following power series, with $X_\ell$ as in \eqref{equ:Xell}$:$
    \begin{align}\label{equ:Agdef}
        \sum_g A_g \hbar^g
        :=& 
        -P_1 \left(  -\frac \hbar {P_1}
        +
        \sum_{\ell\geq 1} \frac{\mu(\ell)}{\ell}
        \left(
        -
        \sum_{j\geq 1} \frac{ (-1)^j}{j} X_\ell^j
        -
        \sum_{k=1}^\infty\sum_{j=0}^\infty 
        \frac{B_k}{k} 
        { -k \choose j }
        \left(\frac{\ell \hbar^\ell}{P_\ell}\right)^k
        X_\ell^j
        \right) \right);
\\
\label{equ:Cgdef}
        \sum_g C_g \hbar^g
        :=& 
        -\frac{P_1}2 \left[ 
        \left(  -\frac \hbar {P_1}
        +
        \sum_{\ell\geq 1} \frac{\mu(\ell)}{\ell}
        \left(
        -
        \sum_{j\geq 1} \frac{ (-1)^j}{j} X_\ell^j
        -
        \sum_{k=1}^\infty\sum_{j=0}^\infty 
        \frac{B_k}{k} 
        { -k \choose j }
        \left(\frac{\ell \hbar^\ell}{P_\ell}\right)^k
        X_\ell^j
        \right)
        \right)^2
        \right.
        \\ &
        \left.
          + 
          \sum_{\ell\geq 1}  \frac{\mu(\ell)}{\ell}
          \left( 
            -
          \sum_{j\geq 1} \frac{ (-1)^j}{j} X_{2\ell}^j
          -
          \sum_{k=1}^\infty\sum_{j=0}^\infty 
          \frac{B_k}{k} 
          { -k \choose j }
          \left(\frac{2\ell \hbar^{2\ell}}{P_{2\ell}}\right)^k
          X_{2\ell}^j
            \right)
        \right.
        \nonumber  \\ \nonumber &
        \left.
            -
            \sum_{\ell\geq 1} \left(\frac{\mu(\ell)}{\ell}\right)^2 
            \sum_{k=1}^\infty \sum_{j=0}^\infty B_{k-1} 
          { -k \choose j }
          \left(\frac{\ell \hbar^\ell}{P_\ell}\right)^{k}
          X_\ell^j
               -\frac{\hbar^2}{P_1^2}
        \right]\, .
        \end{align}
Then $A_g$ is a finite linear combination of monomials
\[
\frac{P_{d_1}^{b_1}P_{d_2}^{b_2}\cdots P_{d_s}^{b_s}}{P_{\ell}^{c} } \mbox{ \ \ such that \ \ } \left\{\begin{array}{l} d_i | \ell,  \mbox{ for } 1 \leq i \leq s, \\ b_1 + \cdots + b_s \leq c + 1, \\ g+ b_1d_1 + \cdots + b_sd_s = \ell c + 1. \end{array} \right.
\] 
Similarly, $C_g$ is a finite linear combinations of monomials
\[
\frac{P_{d_1}^{b_1}P_{d_2}^{b_2}\cdots P_{d_s}^{b_s}}{P_{\ell_1}^{c_1} P_{\ell_2}^{c_2} } \mbox{ \ \ such that \ \ } \left\{\begin{array}{l} d_i | \ell_1 \mbox{ or } d_i | \ell_2,  \mbox{ for } 1 \leq i \leq s, \\ b_1 + \cdots + b_s \leq c_1 + c_2 + 1, \\ g+ b_1d_1 + \cdots + b_sd_s = \ell_1 c_1 + \ell_2 c_2 + 1. \end{array} \right.
\]
Furthermore, the equivariant Euler characteristics $\chi_2^\bbS(\MM_{g,n})$ are expressed as follows:
\begin{align*}
    \chi^\bbS_2(\MM_{0,\bullet})
    &= 
    -\frac {P_1} 2 \left[ \left( \sum_{\ell\geq 1} \frac{\mu(\ell)}{\ell} \log P_\ell \right)^2
    + \sum_{\ell\geq 1} \frac{\mu(\ell)}{\ell} \log P_{2\ell} \right]
    -P_1
    +\frac 12 P_1^2 +\frac 1 2P_2
    \\
    \chi^\bbS_2(\MM_{1,\bullet})
    &=
    +\frac {P_1} 2 \left[ \left( \sum_{\ell\geq 1} \frac{\mu(\ell)}{\ell} \log P_\ell \right)^2
    + \sum_{\ell\geq 1} \frac{\mu(\ell)}{\ell} \log P_{2\ell} \right]
    +
    \left( \sum_{\ell\geq 1} \frac{\mu(\ell)}{\ell} \log P_\ell \right)
    A_1
    + C_1 -1
    \\
    \chi^\bbS_2(\MM_{2,\bullet})
    &=  
    \left( \sum_{\ell\geq 1} \frac{\mu(\ell)}{\ell} \log P_\ell \right)\left( A_2-A_{1} \right)
     +  C_2 -C_1 + P_1 \\
    \chi^\bbS_2(\MM_{g,\bullet})
     &=  
     \left( \sum_{\ell\geq 1} \frac{\mu(\ell)}{\ell} \log P_\ell \right)\left( A_g-A_{g-1} \right)
      +  C_g -C_{g-1} \quad\quad\text{for $g\geq 3$}
  \end{align*}
\end{cor}
We remark that $X_1=0$, and hence the terms involving positive powers of $X_1$ can be dropped from the sums above.  Furthermore, the individual terms appearing in the formulas for the Euler characteristics have natural interpretations in terms of graphs. See \S \ref{sec:terms discussion}, for details.
The first few Laurent polynomials $A_g$ and $C_g$ are: 
\begin{align*}
  A_1 &= -\frac{P_1^2-P_2}{2P_2}, \quad \quad 
  C_1 = \frac{P_1^2+P_2}{2P_2}, \quad \quad 
  A_2 = \frac 1{12P_1}-\frac{P_1^3}{4P_2^2}+\frac{P_1}{2P_2}-\frac{P_1^2}{3P_3}, \quad \quad 
  C_2 = \frac1{8P_1} +\frac{P_1^3}{8P_2^2}-\frac{P_1P_2}{4P_4},\\[1em]
  A_3 &= -\frac{P_1^4}{6P_2^3}+\frac{P_1^2}{2P_2^2}+\frac{P_1}{2P_3}+\frac{P_1P_3}{6P_6}, \quad \quad 
  C_3 = \frac{1}{24P_1^2}+\frac{P_1^4}{24P_2^3}+\frac{P_1^2}{8P_2^2}-\frac5{24P_2}+\frac{P_1}{3P_3}-\frac{P_1^3}{6P_2P_3}-\frac{P_1P_3}{6P_6}, \\[1em]
  A_4 &= -\frac1{120P_1^3}-\frac{P_1^5}{8P_2^4}+\frac{P_1^3}{2P_2^3}-\frac{P_1}{6P_2^2}-\frac{P_1^3}{6P_3^2}-\frac{P_1^2}{5P_5}+\frac{P_1P_2}{6P_6},
  \\
  C_4 &= -\frac1{288 P_1^3}+\frac{P_1^5}{96 P_2^4}+\frac{5P_1^3}{24P_2^3}-\frac{7P_1}{16P_2^2}-\frac{1}{24P_1P_2}-\frac{P_1^3}{18 P_3^2}-\frac{2}{9P_3}+\frac{P_1^4}{12P_2^2P_3}+\frac{5P_1^2}{12P_2P_3}
  -\frac{P_1P_2^2}{8P_4^2} \\&\quad+\frac{P_1}{2P_4}-\frac{P_1P_2}{6P_6}-\frac{P_3}{12P_6}+\frac{P_1^2P_3}{12P_2P_6},
  \\[1em]
  A_5 &= -\frac{P_1^6}{10 P_2^5}+\frac{P_1^4}{2 P_2^4}-\frac{P_1^2}{6 P_6}+\frac{P_1^2}{2
  P_3^2}-\frac{P_1^2}{3 P_2^3}+\frac{P_1}{2 P_5}+\frac{P_5 P_1}{10 P_{10}},
  \\
  C_5 &= -\frac{P_1^6}{240 P_2^5}-\frac{P_1^5}{18 P_2^3 P_3}-\frac{P_1^4}{12 P_2 P_3^2}+\frac{13
  P_1^4}{48 P_2^4}+\frac{7 P_1^3}{24 P_2^2 P_3}-\frac{P_1^3}{10 P_2 P_5}+\frac{P_3
  P_1^3}{24 P_2^2 P_6}+\frac{11 P_1^2}{36 P_6}+\frac{5 P_1^2}{12 P_3^2}-\frac{53 P_1^2}{72
  P_2^3} \\&\quad -\frac{P_1}{4 P_2 P_3}
  +\frac{P_1}{5 P_5}-\frac{P_3 P_1}{12 P_2 P_6}-\frac{P_5
  P_1}{10 P_{10}}-\frac{P_2}{12 P_6}+\frac{1}{24 P_2^2}-\frac{1}{24 P_3 P_1}-\frac{P_3}{72
  P_6 P_1}-\frac{1}{240 P_2 P_1^2}-\frac{1}{80 P_1^4}  
  \, .
\end{align*}


\subsection{Relations to the topological and weight zero Euler characteristics}

For fixed genus $g$, by work of Gorsky, the generating function $\sum_n \chi^\mathbb{S} (\MM_{g,n})$ is a finite sum of Laurent monomials of degree $2-2g$ in the inhomogeneous power sum symmetric functions $P_d$, each of whose denominators is a pure power $P_\ell^c$, satisfying conditions analogous to \eqref{equ:P monomial} \cite{Gorsky}.  The coefficients involve certain orbifold Euler characteristics $\chi^{orb}(\MM_{h,s})$, and these have natural expressions in terms of Bernoulli numbers, by the work of Harer and Zagier \cite{HarerZagier86}. 

Gorsky's method of proof is beautiful and elementary: one stratifies $\MM_g$ according to the topological type of the action of the automorphism group and considers the action of automorphisms on the configuration spaces of $n$ points. Each topological type of a curve with an automorphism contributes to precisely one monomial of the special form described above.  The final formula is then deduced using the basic properties of compactly supported topological Euler characteristic, namely that it is additive for stratifications and multiplicative for fibrations.

Weight graded Euler characteristics are more subtle than topological Euler characteristics.  In particular, although weight graded Euler characteristics of algebraic varieties and Deligne-Mumford stacks are additive with respect to stratifications, they are not multiplicative with respect to algebraic fibrations (flat families), even when the fibers are smooth and proper.  Nevertheless, all of the essential ideas in Gorsky's computation can be adapted to compute the $\bbS_n$-equivariant weight zero Euler characteristic of $\MM_{g,n}$, using the interpretation of weight zero cohomology as the singular cohomology of a moduli space of stable tropical curves \cite{CFGP}. As a result, the weight zero Euler characteristic has a structure closely analogous to that of Gorsky's formula: for fixed genus, the generating function is a finite sum of monomials of degree $1-g$ in the inhomogeneous power sum symmetric functions $P_d$, each of whose denominators is a pure power, satisfying conditions analogous to \eqref{equ:P monomial}.  Roughly speaking, each pair of a graph with an automorphism contributes to precisely one monomial of the prescribed special form; see \cite[Proposition~3.2]{CFGP}.  In the final formula, the coefficients are given by certain orbifold Euler characteristics of spaces of graphs, which have natural expressions in terms of Bernoulli numbers; see \cite[Lemma~7.4]{CFGP}, which is proved by induction starting from the $n = 0$ case, due to Kontsevich \cite{Gerlits, K3}.  

We do not have such a topological interpretation for weight two cohomology of moduli spaces of curves, and the weight two Euler characteristic is of a somewhat different form. In particular, for fixed genus $g \geq 2$, the generating function $\sum_{n} \chi_2^{\bbS} (\MM_{g,n})$ is not a finite sum of Laurent monomials in the inhomogeneous power sum symmetric functions. Lacking a topological interpretation for the weight two Euler characteristic, we express it as the Euler characteristic of a complex of decorated graphs, and use basic properties of symmetric sequences and an operadic identity (Proposition~\ref{prop:iHom Gr}) to produce the formula for the all genus generating function stated as Theorem~\ref{thm:eulerchar}. Specializing to a fixed genus, we obtain an expression involving logarithmic terms as well as Laurent monomials in the inhomogeneous power sum symmetric functions, of two different degrees, with denominators that are products of either one or two pure powers (Corollary~\ref{cor:eulerchar}).  



\subsection{Numerical computations}

The generating function in Theorem~\ref{thm:eulerchar} can be implemented to compute the equivariant Euler characteristic of $\gr_2H_c^\bullet(\MM_{g,n})$ for small values of $g$ and $n$, as shown in Figure~\ref{fig:eulerchartable}.
\begin{figure}[h!]
\noindent\scalebox{.85}{
  \begin{tabular}{|g|M|M|M|M|M|M|} \hline \rowcolor{Gray} g,n & 0 & 1 & 2 & 3 & 4 & 5\\
  \hline
  0 & $ 0 $ & $ 0 $ & $ 0 $ & $ 0 $ & $ s_{4} $ & $ -s_{32} $ \\  
  \hline
   1 & $ 0 $ & $ s_{1} $ & $ 0 $ & $ 0 $ & $ -s_{211} $ & $ s_{311} + s_{32} + s_{41} $ \\  
  \hline
   2 & $ 0 $ & $ 0 $ & $ 0 $ & $ -s_{21} - s_{3} $ & $ s_{4} $ & $ -s_{2111} - s_{221} - s_{311} + 2s_{41} + 2s_{5} $ \\  
  \hline
   3 & $ 0 $ & $ s_{1} $ & $ s_{11} + s_{2} $ & $ -2s_{3} $ & $ 2s_{22} - s_{31} - 2s_{4} $ & $ -3s_{2111} - s_{221} - 6s_{311} + 2s_{32} + 2s_{5} $ \\  
  \hline
   4 & $ 0 $ & $ 0 $ & $ 0 $ & $ 2s_{111} $ & $ 3s_{1111} + s_{211} - 3s_{22} - 4s_{31} - 3s_{4} $ & $ 5s_{11111} + 2s_{2111} + 2s_{221} - 6s_{311} + s_{32} - s_{41} + 3s_{5} $ \\  
  \hline
   5 & $ 0 $ & $ s_{1} $ & $ s_{11} + 3s_{2} $ & $ 5s_{21} + 2s_{3} $ & $ -s_{1111} + 5s_{211} - 4s_{31} - 9s_{4} $ & $ 2s_{11111} + 10s_{2111} + 11s_{221} - 14s_{41} - 8s_{5} $ \\  
  \hline
   6 & $ 0 $ & $ -3s_{1} $ & $ -4s_{11} - s_{2} $ & $ -2s_{111} + 8s_{3} $ & $ 4s_{1111} + 8s_{211} - 6s_{22} + 10s_{31} + 3s_{4} $ & $ 6s_{11111} + 21s_{2111} - 4s_{221} + 15s_{311} - 26s_{32} - 22s_{41} - 19s_{5} $ \\  
  \hline
   7 & $ 0 $ & $ s_{1} $ & $ 2s_{11} $ & $ -7s_{111} + 4s_{21} + 3s_{3} $ & $ -15s_{1111} + 7s_{211} + 15s_{22} + 28s_{31} + 9s_{4} $ & $ -25s_{11111} + s_{2111} - 7s_{221} + 35s_{311} - 22s_{32} - 25s_{41} - 37s_{5} $ \\  
  \hline
   8 & $ -s_{} $ & $ -4s_{1} $ & $ -8s_{11} - 9s_{2} $ & $ -6s_{111} - 18s_{21} + 4s_{3} $ & $ -8s_{1111} - 28s_{211} - 10s_{22} + 24s_{31} + 42s_{4} $ & $ -19s_{11111} - 19s_{221} + 91s_{311} + 22s_{32} + 87s_{41} + 20s_{5} $ \\  
  \hline
   9 & $ 4s_{} $ & $ 9s_{1} $ & $ 8s_{11} + s_{2} $ & $ -2s_{111} + s_{21} - 17s_{3} $ & $ -44s_{1111} - 64s_{211} + 17s_{22} - 17s_{31} + 5s_{4} $ & $ -71s_{11111} - 80s_{2111} + 55s_{221} + 110s_{311} + 193s_{32} + 193s_{41} + 69s_{5} $ \\  
  \hline
   10 & $ -4s_{} $ & $ -9s_{1} $ & $ -6s_{11} - 10s_{2} $ & $ 12s_{111} - 46s_{21} - 29s_{3} $ & $ 38s_{1111} - 68s_{211} - 35s_{22} - 79s_{31} + 28s_{4} $ & $ 3s_{11111} - 227s_{2111} - 155s_{221} - 234s_{311} + 50s_{32} + 196s_{41} + 209s_{5} $ \\  
  \hline
  \end{tabular}
  }
\caption{\label{fig:eulerchartable} The equivariant Euler characteristic of $\gr_2H_c^\bullet(\MM_{g,n})$ for $g\leq 10$ and $n\leq 5$, computed using Mathematica and Sage from Theorem \ref{thm:eulerchar} and expressed in terms of Schur functions.}
\end{figure}

Specializing to $n = 0$, and writing out the terms for $g \leq 50$, we find that $\sum_g \chi_2(\MM_{g})\hbar^g$ is equal to:
\begin{equation*} 
\resizebox{.95\hsize}{!}{
$\begin{aligned}
&-\hbar ^8+4 \hbar ^9-4 \hbar ^{10}+4 \hbar ^{11}-7 \hbar ^{12}+8 \hbar ^{13}-16 \hbar
   ^{14}+10 \hbar ^{15}-10 \hbar ^{16}+28 \hbar ^{17}+17 \hbar ^{18}+94 \hbar ^{19}-12
   \hbar ^{20}-382 \hbar ^{21}
   -196 \hbar ^{22}+2181 \hbar ^{23}
   \\&
   -4304 \hbar ^{24}-43135
   \hbar ^{25}+50281 \hbar ^{26}
+737650 \hbar ^{27}-676300 \hbar ^{28}-13670125 \hbar
   ^{29}+13593898 \hbar ^{30}+301176776 \hbar ^{31}
 -303061590 \hbar ^{32}   \\& -7579395579 \hbar
   ^{33}+7577048267 \hbar ^{34}+215276816193 \hbar ^{35}
      -215186827379 \hbar
   ^{36}
-6867916230403 \hbar ^{37}+6867642573519 \hbar ^{38} \\&  +244534892486924 \hbar
   ^{39}-244525053194639 \hbar ^{40}-9660791088309245 \hbar ^{41}+9660393501146309 \hbar
   ^{42}
   +421316847700113027 \hbar ^{43}    \\& -421303860389448419 \hbar ^{44}-20188618782635720913
   \hbar ^{45}+20188137225679260098 \hbar ^{46}+1058435119750859223989 \hbar
   ^{47}
   \\&
   -1058415071189257113479 \hbar ^{48}-60478235401984833358102 \hbar
   ^{49}+60477318015911247931156 \hbar ^{50}+O\left(\hbar ^{51}\right)\, .
    \end{aligned}$
        }
\end{equation*}


\noindent For $n=0$ and $g\leq 200$, we display $|\chi_2(\MM_g)|$ on a logarithmic scale, along with a plot of the sign of $\chi_2(\MM_{g})$, which seems to be eventually periodic.  See Figure~\ref{fig:euler plots}.

\begin{figure}[h!]
\begin{tabular}{cc}
  \includegraphics[width=.45\textwidth]{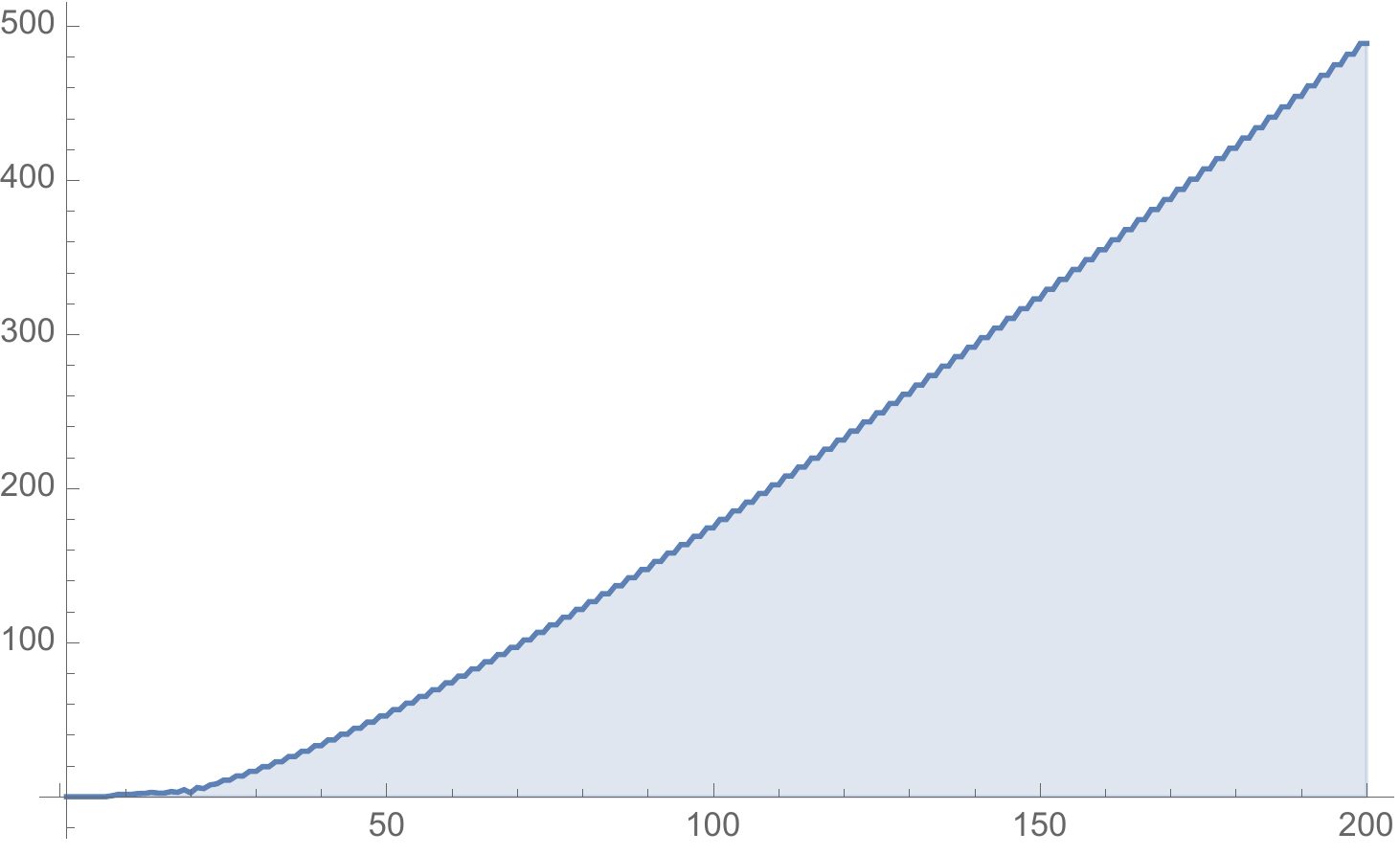}
&
  \includegraphics[width=.45\textwidth]{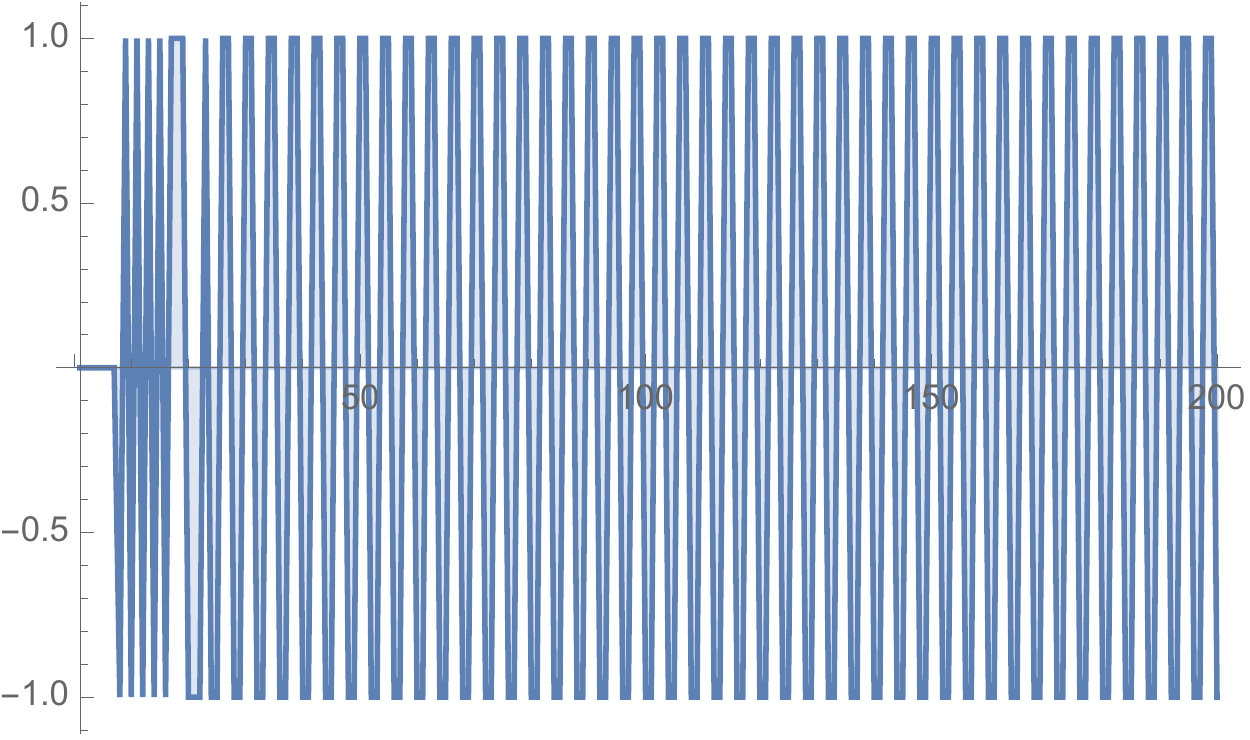}
\end{tabular}
\caption{\label{fig:euler plots} Plots of $\log(1+|\chi_2(\MM_{g})|)$ and the sign of $\chi_2(\MM_{g})$ for $g \leq 200$, on the left and right, respectively.}
\end{figure}

\begin{conjecture}
For $g \geq 23$, the sign of $\chi_2(\MM_g)$ is $-1$ for $g \equiv 0,1 \pmod 4$ and $+1$ for $g \equiv 2, 3 \pmod 4$.
\end{conjecture}

\section{Preliminaries}

\subsection{Notation and conventions} 
We work over the rational numbers $\Q$. All vector spaces are understood to be $\Q$-vector spaces, and likewise all homology and cohomology groups are taken with $\Q$-coefficients. If $V$ is a vector space with the action of a finite group $G$, we identify the invariant subspace $V^G$ with the coinvariant space $V_G$, by averaging over the group action. The phrase differential graded is abbreviated dg.
For a dg vector space $V$ and an integer $k$, let $V[k]$ be the graded vector space obtained by shifting all degrees by $k$, i.e., the degree $j$ part of $V$ is the degree $j-k$ part of $V[k]$.

\subsection{Symmetric sequences and products}\label{sec:symm seq}
A symmetric sequence is a collection $\COp = \{\COp(r)\}_{r\geq 0}$ of (right) $\bbS_r$-modules in dg vector spaces. 
Similarly, a bisymmetric sequence is a collection $\BOp = \{\BOp(r,s)\}_{r,s\geq 0}$ of $\bbS_r\times \bbS_s$-modules. We say that $\COp(r)$ is the part of $\COp$ of \emph{arity} $r$, and $\BOp(r,s)$ is the part of $\BOp$ of \emph{arity} $(r,s)$.

The category of symmetric sequences $\Seq$ has several monoidal products. 
First, for two symmetric sequences $\COp, \DOp$ we have the symmetric sequence $\COp\otimes \DOp$, given by 
\[
 (\COp\otimes \DOp)(r) = \COp(r)\otimes \DOp(r),
\]
with the diagonal $\bbS_r$-action. Second, we have the symmetric sequence $\COp\boxtimes \DOp$ given by  
\[
   ( \COp\boxtimes \DOp)(r) = \oplus_{j+k=r} \Ind_{\bbS_j\times \bbS_k}^{\bbS_r}\COp(j)\otimes \DOp(k).
\]
Let us note in particular that the symmetric sequence
\[
    \COp^{\boxtimes k} = \underbrace{\COp\boxtimes\cdots\boxtimes \COp}_{k\times}
\]
carries an action of $\bbS_k$ by permuting the factors. In other words, $\{(\COp^{\boxtimes k})(r)\}_{k\geq 0,r\geq 0}$ is a bisymmetric sequence.

We then have a third monoidal product on symmetric sequences, the plethysm product $\DOp \circ \COp$, given by 
\[
(\DOp \circ \COp)(r) = \oplus_{k\geq 0} \DOp(k)\otimes_{\bbS_k} \COp^{\boxtimes k}(r).
\]
The category of symmetric sequences $\Seq$ with the plethysm product is monoidal, with unit $\opOne$ given by
\[
    \opOne(r) = \begin{cases}
        \Q &\text{if $r=1$}, \\
        0 &\text{otherwise.}
        \end{cases}
\]

Finally, given $\COp$ and $\DOp$ symmetric sequences of dg vector spaces we define the dg vector space 
\[
  \COp\otimes_\bbS \DOp := \bigoplus_{r\geq 0} \COp(r) \otimes_{\bbS_r} \DOp(r).
\]

\subsection{Equivariant Euler characteristics}
For $F$ an endomorphism of a finite dimensional graded vector space $V=\oplus_i V_i$, we denote by $\STr(F)$ the super-trace of $F$, that is, the alternating sum
\[
\STr(F) = \sum_{i} (-1)^i \mathrm{Tr}(F|_{V_i}).
\]
For $\sigma\in \bbS_r$ a permutation we denote by $i(\sigma)=(i_1,i_2,\dots)$ its cycle type, with $i_k$ the number of cycles of length $k$.
For $\rho$ a graded $\bbS_r$-representation on $V$, we denote the $\bbS_r$-equivariant Euler characteristic in $\Lambda =\Q\llbracket p_1,p_2,\dots\rrbracket$ by
\[
\chi^{\bbS}(V) = \frac 1 {r!} \sum_{\sigma \in \bbS_r} \STr(\rho(\sigma)) p^{i(\sigma)}, 
\]
 using the multi-index notation $p^{i(\sigma)}=p_1^{i_1(\sigma)}p_2^{i_2(\sigma)}\cdots$.
If $V$ has an additional (non-negative) grading with grading generator $G$ (i.e., where $G$ is equal to $k$ times the identity on the subspace of degree $k$) we use the notation 
\begin{equation}
  \label{equ:chiudef}
  \chi^u(\AOp)=  \frac 1 {r!} \sum_{\sigma \in \bbS_r} \STr(u^G \rho(\sigma)) p^{i(\sigma)} \in \Lambda\llbracket u \rrbracket.
\end{equation}
We freely generalize this notation to multiple gradings and other variables counting the multidegree. For example, if $V$ has an additional $\Z^2$-grading, with grading generators $G_1$, $G_2$, then we write
\[
  \chi^{u,v}(\AOp)=  \frac 1 {r!} \sum_{\sigma \in \bbS_r} \STr(u^{G_1} v^{G_2} \rho(\sigma)) p^{i(\sigma)} \in \Lambda\llbracket u,v \rrbracket,
\] 
and similarly for a $\Z^3$-grading.

For $\AOp$ a symmetric sequence such that $\AOp(r)$ is finite dimensional for each $r$, we write
\[
\chi^{\bbS}(\AOp) := \sum_{r\geq 0} \chi^{\bbS}(\AOp(r))
\]
and similarly for $\chi^u(\AOp)$, when $\AOp$ has an additional grading.  

The operations on symmetric sequences discussed in \S \ref{sec:symm seq} induce operations on symmetric functions. First, 
\[
\chi^\bbS(\AOp\boxtimes \BOp)= \chi^\bbS(\AOp) \chi^\bbS(\BOp).
\]
Second, there is a plethysm product $\circ$ on symmetric functions $\circ: \Lambda\otimes\Lambda\to \Lambda$ such that
\[
\chi^\bbS(\AOp\circ \BOp)= \chi^\bbS(\AOp) \circ \chi^\bbS(\BOp).  
\]
Concretely, the plethysm product on symmetric functions may be computed by the following rules:
\begin{equation}\label{equ:plethism euler 1}
\begin{aligned}
(f_1+f_2)\circ g &= f_1\circ g + f_2\circ g \quad \quad \quad \quad \quad \quad 
(f_1f_2)\circ g = (f_1\circ g)(f_2\circ g) \\
p_n\circ f &= f(p_n,p_{2n},p_{3n},\dots) 
\ \  \text{for $f=f(p_1,p_2,p_3,\dots)$}
\end{aligned}
\end{equation}
In the presence of an additional grading, encoded by a variable $u$ as in \eqref{equ:chiudef}, the last rule is replaced by 
\begin{align}\label{equ:plethism euler 2}
  p_n\circ f &= f(u^n, p_n,p_{2n},p_{3n},\dots) 
\ \  \text{for $f=f(u, p_1,p_2,p_3,\dots)$}
\end{align}
For more details we refer to \cite{Macdonald} or the concise introduction to symmetric functions in \cite{GK}.

On the subspace of $\Lambda$ consisting of symmetric polynomials (finite series), there is an inner product given by 
\begin{equation}\label{equ:prod def}
\langle p^i, p^j \rangle
= 
\delta_{ij} \prod_k k^{i_k} i_k!
= 
\prod_k \left( \left(\frac{\partial}{\partial p_k}\right)^{i_k} k^{j_k}p_k^{j_k}\right)_{p_k=0}
= 
\left( \prod_k \left(k\frac{\partial}{\partial p_k}\right)^{i_k} p^j\right)_{p_1=p_2=\cdots =0}.
\end{equation}
Here, $i=(i_1,\dots)$ and $j=(j_1,\dots)$ are multi-indices, $p^i:=p_1^{i_1}p_2^{i_2}\dots$, and $\delta_{ij}$ is the Kronecker symbol.
This inner product has the property that for $\AOp, \BOp$ finite dimensional symmetric sequences 
\begin{equation}\label{equ:tens prod}
  \chi\left(\AOp\otimes_\bbS \BOp\right) 
= 
  \langle\chi^\bbS(\AOp), \chi^\bbS(\BOp)  \rangle.
\end{equation}

\subsection{Examples of symmetric sequences and their equivariant Euler characteristics}

For later use we shall recall the Euler characteristics of some symmetric sequences from the literature.  

The commutative operad is a simple and natural symmetric sequence, which we consider in two variations:
\begin{align*}
\Com_1(r) &= \Q \quad \text{for $r\geq 0$},
&
\Com(r) &=
\begin{cases}
   \Q & \text{for $r\geq 1$}, \\
   0 & \text{for $r=0$}.
\end{cases},
\end{align*}
Here, $\Q$ denotes the trivial $\bbS_r$-representation.
The corresponding equivariant Euler characteristics are \cite[\S 5.1]{Getzler}:
\begin{align}\label{equ:euler Com}
    \chi^\bbS(\Com_1) &= \exp\left(\sum_{\ell}\frac{p_\ell}{\ell} \right) &
    \chi^\bbS(\Com) &= \exp\left(\sum_{\ell}\frac{p_\ell}{\ell}\right) -1.
\end{align}

The Lie operad is the symmetric sequence defined such that $\Lie(r)$ is the subspace of the free Lie algebra $\FreeLie(x_1,\dots,x_r)$ in $r$ letters consisting of linear combinations of Lie words in which each $x_j$ occurs exactly once.
We shall also need a degree shifted version $\Lie_0$ defined in the same way, except that the generators $x_j$ are considered of cohomological degree one, and one has an overall degree shift by $-1$, i.e.,
\[
  \Lie_0(r) = \Lie(r) \otimes \sgn_r[1-r].
\]
The corresponding equivariant Euler characteristics are
\begin{align}\label{equ:euler Lie}
\chi^\bbS(\Lie) &= - \sum_{\ell \geq 1} \frac{\mu(\ell)} \ell \log(1-p_\ell), &
\chi^\bbS(\Lie_0) &=  \sum_{\ell \geq 1} \frac{\mu(\ell)} \ell \log(1+p_\ell).
\end{align}
See \cite[Proposition 5.3 and Lemma 5.4]{Getzler}.

The Poisson operad is the symmetric sequence $\Pois=\Com\circ\Lie$.
The elements of $\Pois(r)$ can be understood as linear combinations of Poisson words in letters $x_1,\dots,x_r$ such that each letter occurs exactly once.
We shall also use a degree shifted version 
\[
\Pois_0 = \Com\circ\Lie_0 \, .  
\]
Both $\Pois$ and $\Pois_0$ carry a \emph{complexity grading} by the number of Lie brackets.
More precisely, we define $\Lie(r)$ and $\Lie_0(r)$ to be concentrated in complexity $r-1$, and this descends to give the complexity grading on $\Pois$ and $\Pois_0$.
By  \cite[Proposition 5.5]{Getzler} and \cite[Proposition 6.3]{AroneTurchin}, the graded equivariant Euler characteristics, with the variable $u$ counting the complexity as in \eqref{equ:chiudef}, are:
\begin{align}
  \label{equ:poisn euler}
    \chi^u(\Pois) &= \prod_{l\geq 1} (1- u^\ell p_\ell)^{ \frac 1 \ell \sum_{d\mid \ell}\mu(\ell/d) \frac 1 {u^d} } -1, \\
    \label{equ:poisn0 euler}
    \chi^u(\Pois_0) &= \prod_{\ell\geq 1} (1+ u^\ell p_\ell)^{ \frac 1 \ell \sum_{d\mid \ell}\mu(\ell/d) \frac 1 {u^d} } -1 
    \, . 
\end{align}
Also note that the complexity grading on $\Pois_0$ is, by definition, equal to the cohomological grading.

Next consider the symmetric sequence concentrated in arity one 
\[
HS(r) = 
\begin{cases}
  H(S^1) = \Q\oplus \Q[-1] &\text{for $r=1$} \\
  0 &\text{for $r\neq 1$}
\end{cases}.
\]
Then one defines the symmetric sequence 
\[
\BV_0 := \Pois_0\circ HS.
\]
This is a version of the Batalin-Vilkovisky operad (see \cite[\S 13.7]{LodayVallette}), hence the notation.
We declare the complexity grading on $\BV_0$ to be the same as the cohomological grading.

It is clear that $\chi^u(HS)=p_1-up_1$. From  \eqref{equ:plethism euler 1}, \eqref{equ:plethism euler 2}, and \eqref{equ:poisn0 euler} one then finds
\begin{align}
    \label{equ:bv euler}
    \chi^u(\BV_0) &= \prod_{\ell\geq 1} (1+ u^\ell (1-u^\ell) p_\ell)^{ \frac 1 \ell \sum_{d\mid \ell}\mu(\ell/d) \frac 1 {u^d} } -1 .
\end{align}
Elements of $\BV_0(r)$ are identified with linear combinations of Poisson words in variables $x_1,\dots,x_r, Dx_1,\dots, Dx_r$, with every index occurring exactly once, for example
\[
x_1\wedge Dx_2\wedge [x_3,Dx_4] \in \BV_0(4).  
\]
For technical reasons we also define the reduced sub-symmetric sequence $\croBV_0\subset \BV_0$ to be spanned by those Poisson words that do not have a factor $x_j$ for any $j$.  In other words,  
\[
  \BV_0 = (\croBV_0 \boxtimes \Com_1) \oplus \Com.
\]
Hence 
\begin{equation}\label{equ:euler crBV}
  \chi^u(\croBV_0)
  =
  (\chi^u(\BV_0)+1)\chi(\Com_1)^{-1} -1
  = 
  \prod_{\ell\geq 1} e^{-\frac 1 \ell p_\ell }(1+ u^\ell (1-u^\ell) p_\ell)^{ \frac 1 \ell \sum_{d\mid \ell}\mu(d/\ell) \frac 1 {u^d} } -1 .
\end{equation}

Finally, note that the definition of the $\bbS$-equivariant Euler characteristic above has a natural generalization for bisymmetric sequences. We shall only need the following example. Define the bisymmetric sequence $\DS$ such that 
\begin{equation}\label{equ:DSdef}
  \DS(r,s) = 
  \begin{cases}
    \Q[\bbS_r] & \text{for $r=s$} \\
    0 & \text{otherwise}
  \end{cases}.
\end{equation}
Here the group ring $\Q[\bbS_r]$ is considered as an $\bbS_r\times\bbS_r$-module by left and right multiplication. We then have 
\[
\chi^\bbS(\DS)
:= 
\sum_{r,s}\frac 1 {r!s!}\sum_{\sigma\in \bbS_r \atop \nu \in \bbS_s}
\STr_{\Delta_0}(\rho(\sigma, \nu)) p^{i(\sigma)}q^{i(\nu)}
=
\exp\left( \sum_{n\geq 1} \frac 1 n p_n q_n \right) \in \Lambda_{p,q}:= \Q\llbracket p_1,q_1,p_2,q_2,\dots \rrbracket.
\]


\section{Graph complexes associated to moduli spaces of curves}
\subsection{The graph complexes $G^{(g,n)}$ and weight zero compactly supported cohomology of $\MM_{g,n}$}

Let $g$ and $n$ be non-negative integers. The graph complex $G^{(g,n)}$ studied in \cite{CGP2, PayneWillwacher} is generated by connected graphs of genus $g$ with no loop edges and no vertices of valence 2.  The vertices of valence 1 are called \emph{external}, and the other vertices are \emph{internal}.  Each generator has precisely $n$ external vertices, labeled in bijection with $\{ 1, \ldots, n \}$, and the symmetric group $\bbS_n$ acts by permuting these labels.  The edges not connecting to univalent vertices are called \emph{structural}, and the cohomological degree of a generator is the number of structural edges.  Each generator comes with a total ordering of its structural edges, and we impose the relation that reordering is multiplication by the sign of the induced permutation.  The differential $\delta_{split}$ on $G^{(g,n)}$ is defined by splitting internal vertices.
\begin{align}
  \delta_{split} \Gamma &= \sum_{v \text{ vertex} } 
  \Gamma\text{ split $v$} 
  &
  \begin{tikzpicture}[baseline=-.65ex]
  \node[int] (v) at (0,0) {};
  \draw (v) edge +(-.3,-.3)  edge +(-.3,0) edge +(-.3,.3) edge +(.3,-.3)  edge +(.3,0) edge +(.3,.3);
  \end{tikzpicture}
  &\mapsto
  \sum
  \begin{tikzpicture}[baseline=-.65ex]
  \node[int] (v) at (0,0) {};
  \node[int] (w) at (0.5,0) {};
  \draw (v) edge (w) (v) edge +(-.3,-.3)  edge +(-.3,0) edge +(-.3,.3)
   (w) edge +(.3,-.3)  edge +(.3,0) edge +(.3,.3);
  \end{tikzpicture}
  \end{align}
  
  \begin{thm}[\cite{CGP2}]
  For each $g,n$ such that $2g+n\geq 3$ there is an $\bbS_n$-equivariant isomorphism
  \[
   H(G^{(g,n)}) \cong W_0H_c(\MM_{g,n}) 
  \]
  with the weight zero part of the compactly supported cohomology of the moduli space of curves.
  \end{thm}

\noindent The differential does not affect the Euler characteristic. In this paper we will work with $G^{(g,n)}$ and its variants primarily as graded representations of symmetric groups. 

  \subsection{Two variants of $G^{(g,n)}$}
 Distinguishing structural edges from nonstructural edges in the definition of $G^{(g,n)}$ will be inconvenient for the calculations that follow. For this reason, we define the variant $\tG^{(g,n)}$ exactly as above, except that the degree of a graph is the number of edges, and the orientation is given by a total ordering of all of the edges, not just the structural ones. 
  For $(g,n)\neq (0,2)$ one can readily see that
  \[
  \tG^{(g,n)} \cong G^{(g,n)} \otimes \sgn_n[-n].  
  \]
  
  We also consider an enlargement of $\tG^{(g,n)}$ obtained by relaxing the connectedness assumption. Generators of the resulting graph complex $\ftG^{(g,n)}$ are possibly disconnected graphs of genus $g$ with no loop edges and no vertices of valence 2. The external vertices are labeled in bijection with $\{1, \ldots, n\}$, and the orientation is given by a total ordering of all edges.  The genus $g$ of a possibly disconnected graph with $e$ edges and $v$ vertices is 
  \[
  g := e - v + 1. 
  \]
Note that some disconnected graphs have negative genus.

From $\ftG^{(g,n)}$, we build the symmetric sequence $\ftG$ defined by 
  \[
  \ftG(r) := \bigoplus_{g\in \Z} \ftG^{(g,r)}.
  \]
On this symmetric sequence we consider two different gradings: the \emph{genus grading}, in which $\ftG^{(g,r)}$ is of genus $g$, and the \emph{complexity grading} in which $\ftG^{(g,r)}$ is of complexity $g+r-1$. This is the genus of the connected graph obtained by fusing all external vertices into one.  Note that the complexity grading is non-negative, and the number of isomorphism classes of graphs of a given complexity is finite.

\subsection{Graph complexes with one special vertex}
Let $\AOp$ be a symmetric sequence. Consider 
\[
\AOp \otimes_\bbS \ftG.    
\]
This is a graded vector space whose elements are naturally interpreted as linear combinations of graphs as above, but with one special vertex that carries a decoration in $\AOp$.
\[
\begin{tikzpicture}
\node[int] (w1) at (-.8, .7) {};
\node[int] (w2) at (-.8, 1.7) {};
\node[int] (w3) at (.8, .7) {};
\node[int] (w4) at (.8, 1.7) {};
\node[draw, ellipse, inner sep=1em] (b) at (0,-.7) {$a$};
\draw (w1) edge (w2) edge (w3) edge (w4)
(w2) edge (w3) edge (w4) 
(w3) edge (w4)
(b.north west) edge (w1) 
(b) edge (w2) (b.north east) edge (w3);
\draw [decorate,
    decoration = {brace}] (1.5,2) --  (1.5,0.3)
    node[pos=0.5,left=-30pt,black]{$\in\ftG$};
\draw [dashed] (-1.2,0.2)--(1.2,0.2);
    \node at (-1.5,0.2) {$\otimes_\bbS$};
\draw [decorate,
    decoration = {brace}] (1.5,0.1) --  (1.5,-1.2)
    node[pos=0.5,left=-30pt,black]{$\in\AOp$};
\end{tikzpicture}
\]
More precisely, the decoration of the special vertex is in $\AOp(r)$, with $r$ the valence of the special vertex.
The complexity grading on $\ftG$ induces a complexity grading on $\AOp \otimes_\bbS \ftG$, in which the complexity of a graph with one vertex decorated by $\AOp$ is equal to its genus. The following result gives a general formula for the complexity-graded Euler characteristic of such $\AOp$-decorated graph complexes.

\begin{prop}\label{prop:iHom Gr}
  Let $\AOp$ be any symmetric sequence in finite dimensional graded vector spaces.
  Then we have
  \begin{align*}
  \chi^u\left(\AOp\otimes_\bbS \ftG)\right)
  &=
  \chi^u\left((\AOp\circ \Lie_0^*) \otimes_{\bbS} \croBV_0\right)
  =
  \left\langle \chi^\bbS(\AOp)\circ \chi^\bbS(\Lie_0) 
  , \chi^u(\croBV_0)\right\rangle
  \, ,
  \end{align*}
with $u$ counting the complexity, and $\chi^u(-)$ being the graded Euler characteristic. 
\end{prop}

\noindent Note that the graded Euler characteristic is well-defined since the number of graphs of any given complexity is finite.
Note also that the complexity grading is derived solely from the factor $\ftG$, or respectively $\croBV_0$. The other factors $\Lie_0$ or $\AOp$ do not contribute to the complexity, i.e., they are considered to be of complexity zero.
Finally, recall that $\chi^\bbS(\Lie_0)$ and $\chi^u(\croBV_0)$ are given in \eqref{equ:euler Lie} and \eqref{equ:euler crBV}, respectively.


\subsection{Proof of Proposition \ref{prop:iHom Gr}}
\label{sec:iHom Gr proof}

\label{sec:Graphs operad}

In the proof, we will use the Kontsevich graphical operad $\Graphs_0$. We briefly sketch its construction and main properties. For details we refer to Kontsevich's original paper \cite{KMotives}, or recollections elsewhere, for example \cite[\S 2.9f]{FTW2}.
Elements of $\Graphs_0(r)$ are series of graphs with $r$ labeled external vertices and an arbitrary number of internal vertices such as in the following picture.
 \[
   \begin{tikzpicture}
  \node[ext] (v1) at (0,0) {1};
  \node[ext] (v2) at (0.5,0) {2};
  \node[ext] (v3) at (1,0) {3};
  \node[ext] (v4) at (1.5,0) {4};
  \node[ext] (v5) at (2,0) {5};
  \node[int] (w1) at (0.5,.7) {};
  \node[int] (w2) at (1.0,.7) {};
  \draw (v1) to [out=50,in=120] (v2) (v1) edge (w1)  (v2) edge (w1) edge (w2) (v3) edge (w1) edge (w2) (v4) edge (w2) (w1) edge (w2);
 \end{tikzpicture}
\]
The symmetric group $\bbS_r$ acts by permuting the labels $1,\dots,r$ of external vertices.
The graph is oriented by a total ordering of its edgs. The sign conventions are thus identical to those for $\tG^{(g,n)}$ above. The degree is the number of edges.
One allows loops at both internal and external vertices.
Note that the symmetric sequence $\Graphs_0$ 
carries a dg operad structure in which the differential is given by vertex splitting; however, the differential does not matter for the purposes of computing Euler characteristics.

We also consider the variant $\BVGraphs_0$, obtained as a quotient of $\Graphs_0$ by setting to zero graphs with loop edges at internal vertices. (Loops are still allowed at external vertices.)  Both $\Graphs_0$ and $\BVGraphs_0$ have a complexity grading: the complexity of a graph with $v$ internal vertices and $e$ edges is $e-v$.

Furthermore, we consider the subspaces
\begin{align*}
  \croGraphs_0 &\subset \Graphs_0 
  &
  \croBVGraphs_0 &\subset \BVGraphs_0 
\end{align*}
spanned by graphs all of whose external vertices have valence $\geq 1$. 

 \begin{thm}[Kontsevich \cite{KMotives}, Lambrechts-Volic \cite{LambrechtsVolic}\footnote{Only the result for $\Graphs_0$ is stated in the given references, but the version for $\BVGraphs_0$ is an easy consequence.}]
  \label{thm:Graphs qiso}
     There are operad quasi-isomorphisms $\Pois_0\to\Graphs_0$ and $\BV_0\to \BVGraphs_0$ that preserve the complexity gradings and restrict to quasi-isomorphisms on the reduced parts $(-)^\rd$. In particular, 
     \[
        \chi^u(\croBVGraphs_0) = \chi^u(\croBV_0).
     \]
 \end{thm}

The connection between $\croBVGraphs_0$ and our $\ftG$ is as follows.

\begin{lemma}\label{lem:circ adjoint}
Let $\BOp$ be any symmetric sequence and $C\in \Z$. Then 
\begin{equation}\label{equ:circ adjoint}
    \BOp \otimes_\bbS \gr^C \croBVGraphs_0
    \cong 
(\BOp\circ \Com) \otimes_\bbS \gr^C \ftG,
\end{equation}
with $\gr^C(-)$ referring to the part of complexity $C$.
In particular, if $\BOp$ is finite-dimensional in each arity then 
\[
\chi^u\left(\BOp \otimes_\bbS \croBVGraphs_0\right)
=
\chi^u\left((\BOp\circ \Com) \otimes_\bbS \ftG\right).
\] 
\end{lemma}
\begin{proof}
As in the previous section the right-hand side of \eqref{equ:circ adjoint} may be interpreted as the space of linear combinations of graphs of genus (i.e., loop number) $C$, with one special vertex decorated by $\BOp\circ \Com$. Equivalently such graphs may be interpreted as 3-level graphs with one vertex decorated by $\BOp$, connected to another layer of vertices contributed by the factor $\Com$.
\[
\begin{tikzpicture}
\node[int] (w1) at (-.6, .7) {};
\node[int] (w2) at (-.6, 1.7) {};
\node[int] (w3) at (.6, .7) {};
\node[int] (w4) at (.6, 1.7) {};
\draw (w1) edge (w2) edge (w3) edge (w4)
(w2) edge (w3) edge (w4) 
(w3) edge (w4);
\draw [decorate,
    decoration = {brace}] (1.5,2) --  (1.5,0.5)
    node[pos=0.5,left=-30pt,black]{$\in\ftG$};
\node[ext] (v1) at (-1,0) {};
\node[ext] (v2) at (0,0) {};
\node[ext] (v3) at (1,0) {};
\draw (w1) edge (v1) edge (v2) (w3) edge (v3)
(w2) edge (v1) (w4) edge (v2);
\node[draw, ellipse, inner sep=1em] (b) at (0,-1.2) {$b$};
\draw (b) edge (v1) edge (v3) edge (v2);
\draw [dashed] (-1.2,.4)--(1.2,.4);
\node at (-1.5,.4) {$\otimes_\bbS$};
\draw [decorate,
    decoration = {brace}] (1.5,0.3) --  (1.5,-.4)
    node[pos=0.5,left=-37pt,black]{$\in\Com$};
\draw [decorate,
    decoration = {brace}] (1.5,-.5) --  (1.5,-1.7)
    node[pos=0.5,left=-24pt,black]{$\in\BOp$};
\end{tikzpicture}
\]
But the left-hand side of \eqref{equ:circ adjoint} also has the interpretation as linear combinations of three-level graphs of genus $C$. In this case the middle layer of vertices is contributed by the external vertices of graphs in $\croBVGraphs_0$.
\[
\begin{tikzpicture}
\node[int] (w1) at (-.6, .7) {};
\node[int] (w2) at (-.6, 1.7) {};
\node[int] (w3) at (.6, .7) {};
\node[int] (w4) at (.6, 1.7) {};
\draw (w1) edge (w2) edge (w3) edge (w4)
(w2) edge (w3) edge (w4) 
(w3) edge (w4);
\draw [decorate,
    decoration = {brace}] (1.5,2) --  (1.5,-.3)
    node[pos=0.5,left=-70pt,black]{$\in\croBVGraphs_0$};
\node[ext] (v1) at (-1,0) {};
\node[ext] (v2) at (0,0) {};
\node[ext] (v3) at (1,0) {};
\draw (w1) edge (v1) edge (v2) (w3) edge (v3)
(w2) edge (v1) (w4) edge (v2);
\node[draw, ellipse, inner sep=1em] (b) at (0,-1.2) {$b$};
\draw (b) edge (v1) edge (v3) edge (v2);
\draw [dashed] (-1.2,-.4)--(1.2,-.4);
\node at (-1.5,-.4) {$\otimes_\bbS$};
\draw [decorate,
    decoration = {brace}] (1.5,-.5) --  (1.5,-1.7)
    node[pos=0.5,left=-24pt,black]{$\in\BOp$};
\end{tikzpicture}
\]
Hence we conclude that both sides of \eqref{equ:circ adjoint} are identical.
\end{proof}

As a final ingredient we will use the formula 
\begin{equation}\label{equ:Com Koszul}
\chi^\bbS(\Lie_0^* \circ \Com)=
\chi^\bbS(\Lie_0^*) \circ \chi^\bbS(\Com)=
p_1.
\end{equation}
This formula is a consequence of the Koszul property of the commutative operad \cite[Theorem~7.6.4(i) and Proposition~13.1.7]{LodayVallette}. One can also derive it directly from \eqref{equ:euler Com} and \eqref{equ:euler Lie}, using the plethysm rules \eqref{equ:plethism euler 1}.

\begin{proof}[Proof of Proposition \ref{prop:iHom Gr}]
The plethysm product with $p_1$, $f\mapsto f\circ p_1$, is the identity. Hence, from \eqref{equ:Com Koszul} it follows that on the level of Euler characteristics 
\[
\chi^u(\AOp\otimes_\bbS \ftG))
=
\chi^u((\AOp\circ \Lie_0^* \circ \Com) \otimes_\bbS \ftG)).
\]
Using Lemma \ref{lem:circ adjoint} this is the same as 
\[
 \chi^u((\AOp\circ \Lie_0^*) \otimes_\bbS \croBVGraphs_0)).
\]
Using Theorem \ref{thm:Graphs qiso} this in turn is identified with 
\[
    \chi^u((\AOp\circ \Lie_0^*) \otimes_{\bbS} \croBV_0),
\]
Applying (the graded version of) \eqref{equ:tens prod} we can evaluate this further to 
\[
  \chi^u((\AOp\circ \Lie_0^*) \otimes_{\bbS} \croBV_0)
  =
  \left\langle \chi^\bbS(\AOp\circ \Lie_0^*) 
  , \chi^u(\croBV_0)\right\rangle
  =
  \left\langle \chi^\bbS(\AOp)\circ \chi^\bbS(\Lie_0) 
  , \chi^u(\croBV_0)\right\rangle,
\]
and hence Proposition \ref{prop:iHom Gr} follows.

\end{proof}


\subsection{The graph complexes $X_{g,n}$ and weight two compactly supported cohomology of $\MM_{g,n}$}
The main result of \cite{PayneWillwacher} is the identification of $\gr_2 H^\bullet_c(\MM_{g,n})$ with the cohomology of a graph complex $X_{g,n}$ that is close to graph complexes arising in knot theory and the embedding calculus.
We shall use this identification to show our main Theorem \ref{thm:eulerchar}, by computing the Euler characteristic of $X_{g,n}$. 

Recall that the generators for $X_{g,n}$ are simple graphs without loops or multiple edges, in which no vertices have valence 2.  The vertices of valence at least 3 are \emph{internal} and those of valence 1 are \emph{external}.  Each external vertex is decorated with an element from the set $\{\epsilon, \omega, 1, \ldots, n \}$, such that:
\begin{itemize}
\item Each label $1, \ldots, n$ appears exactly once and the label $\omega$ appears exactly twice;
\item The graph obtained by joining all external vertices labeled $\epsilon$ or $\omega$ is connected and has genus $g$;
\end{itemize}
Say that an edge with two external vertices labeled $a$ and $b$ is an \emph{$(a,b)$-edge}. We further require that
\begin{itemize}
\item No connected component is an $(\epsilon,\omega)$ or $(\omega,\omega)$ edge; 
\end{itemize}

\noindent An edge is \emph{structural} if it does not contain an external vertex with label from $\{1, \ldots, n\}$. The degree of a graph is the number of structural edges plus one.  Each generator comes with a total ordering of the structural edges, and we impose the relation that permuting the structural edges is multiplication by the sign of the permutation.

The differential $\delta$ on $X_{g,n}$ is a sum of two parts  
$
\delta = \delta_{split} + \delta_{join},
$ 
defined by splitting internal vertices
\begin{align}
  \delta_{split} \Gamma &= \sum_{v \text{ vertex} } 
  \Gamma\text{ split $v$} 
  &
  \begin{tikzpicture}[baseline=-.65ex]
  \node[int] (v) at (0,0) {};
  \draw (v) edge +(-.3,-.3)  edge +(-.3,0) edge +(-.3,.3) edge +(.3,-.3)  edge +(.3,0) edge +(.3,.3);
  \end{tikzpicture}
  &\mapsto
  \sum
  \begin{tikzpicture}[baseline=-.65ex]
  \node[int] (v) at (0,0) {};
  \node[int] (w) at (0.5,0) {};
  \draw (v) edge (w) (v) edge +(-.3,-.3)  edge +(-.3,0) edge +(-.3,.3)
   (w) edge +(.3,-.3)  edge +(.3,0) edge +(.3,.3);
  \end{tikzpicture}
\end{align}
and joining external vertices
\begin{align}
  \delta_{join} 
 \begin{tikzpicture}[baseline=-.8ex]
 \node[draw,circle] (v) at (0,.3) {$\Gamma$};
 \node (w1) at (-.7,-.5) {};
 \node (w2) at (-.25,-.5) {};
 \node (w3) at (.25,-.5) {};
 \node (w4) at (.7,-.5) {};
 \draw (v) edge (w1) edge (w2) edge (w3) edge (w4);
 \end{tikzpicture} 
 = 
 \sum_{\substack{S\subset \mbox{\small \{$\epsilon$- and $\omega$- legs$\}$} \\ |S|\geq 2 }}  
 \begin{tikzpicture}[baseline=-.8ex]
 \node[draw,circle] (v) at (0,.3) {$\Gamma$};
 \node (w1) at (-.7,-.5) {};
 \node (w2) at (-.25,-.5) {};
 \node[int] (i) at (.4,-.5) {};
 \node (w4) at (.4,-1.3) {$\epsilon$ or $\omega$};
 \draw (v) edge (w1) edge (w2) edge[bend left] (i) edge (i) edge[bend right] (i) (w4) edge (i);
 \end{tikzpicture} \, .
 \end{align}
We refer to \cite{PayneWillwacher} for details. 
%

\begin{thm}[Theorem 1.1 of \cite{PayneWillwacher}]
    \label{thm:weight2}
There is an $\bbS_n$-equivariant isomorphism
\[
H(X_{g,n})\cong \gr_2 H_c(\MM_{g,n})  
\]
for each $(g,n)\neq (1,1)$ with $2g+n\geq 3$.
\end{thm}
To show Theorem \ref{thm:eulerchar} we may hence compute the equivariant Euler characteristic of the graph complexes $X_{g,n}$, that is, we may use the following corollary.

\begin{cor}\label{cor:omega2 and X}
\[
\omega_2 = \hbar p_1 + \sum_{g,n \atop 2g+n\geq 3} \hbar^n \chi^\bbS(X_{g,n}) 
\]
\end{cor}

\subsection{Two variants of $X_{g,n}$}
We shall also consider variant of the complexes $X_{g,n}$ above. 
First, we define the complexes $\tX_{g,n}$ in the same manner as above, except that we treat all edges in the same way, and declare the degree to be the total number of edges and the ordering of edges to be defined on all edges.  In other words,  
\begin{equation}\label{equ:tXX}
\tX_{g,n} = X_{g,n}\otimes \sgn_n[-n-1].  
\end{equation}
These graph complexes assemble naturally into symmetric sequences
\[
 X(r) := \bigoplus_g X_{g,r} \quad \quad \quad \quad \mbox{ and } \quad \quad \quad \quad \tX(r) := \bigoplus_g \tX_{g,r}. 
\]

We consider also the enlargement $\ftX_{g,n}\supset \tX_{g,n}$ obtained by relaxing the connectedness assumption on generating graphs.  In other words, $\ftX_{g,n}$ is defined just as $\tX_{g,n}$ except that the generators are allowed to be possibly disconnected graphs of genus $g$ without vertices of valence 2, with external vertices labeled by symbols $\{1,\dots,n,\epsilon,\omega\}$, such that each label $1,\dots,n$ appears exactly once and every connected component contains at least one external vertex.  The generators are oriented by a total ordering of the edges set.  Note that we allow graphs with any number of $\epsilon$- or $\omega$-legs, with $(\epsilon,\omega)$- and $(\omega,\omega)$-edges, and we also allow disconnected graphs. The complexity of a generator for $\ftX_{g,n}$ is the genus of the graph obtained by fusing all external vertices.  We again assemble the graph complexes $\ftX_{g,n}$ into a symmetric sequence 
\[
  \ftX(r) := \bigoplus_g \ftX_{g,r}.
\]

\begin{lemma}\label{lem:ftX}
We have an isomorphism of symmetric sequences
\[
  \ftX \cong (\DS \boxtimes \Com_1\boxtimes \Com_1) \otimes_\bbS \ftG.
\]
with $\DS$ the bisymmetric sequence \eqref{equ:DSdef}, and with $\boxtimes$ and $\otimes_\bbS$ operating on the first symmetric sequence structure.
\end{lemma}
\begin{proof}
Unpacking the notation, both sides are the same.
\end{proof}

\section{Euler characteristics}
\label{sec:eulerchar}

The goal of this section is to prove Theorem \ref{thm:eulerchar}. The strategy of the proof is to first compute the equivariant Euler characteristic of the enlarged complexes $\ftX_{g,n}$, and then extract from this the Euler characteristic of $\tX_{g,n}$.  We then deduce the Euler characteristic of $X_{g,n}$, and hence that of $\gr_2H_c(\MM_{g,n})$.

The proof uses some elementary but tedious computations with power series arising from the asymptotic expansions of polygamma functions. We perform these computations in \S \ref{sec:aux comp}, in order not to clutter the main argument line of \S \ref{sec:eulerchar proof}. The reader is encouraged to skip \S \ref{sec:aux comp} on the first reading.
 
\subsection{Auxiliary computations of derivatives}
\label{sec:aux comp}

 We will use the following identity involving gamma functions. 

 \begin{prop}[after Songhafouo Tsopm\'en\'e and Turchin \cite{TT2}]\label{prop:Udef}
 Let $X$ and $u$ be formal variables and $\ell\in \mathbb N$.
 Abbreviate 
 \begin{align}
 E_\ell&:= \frac 1 \ell \sum_{d\mid \ell}\mu(\ell/d)\frac 1 {u^d}
 &
 \lambda_\ell &:= u^\ell (1-u^\ell)\ell.
 \end{align}
 Then we have the equality of formal power series in $u$
 \begin{equation}\label{equ:propUdef}
    (1+\p_a)^{X}e^{-a}(1+ \lambda_\ell a)^{E_\ell}\mid_{a=0}
    =
    \frac {(-\lambda_\ell)^X \Gamma(-E_\ell+X) }{\Gamma(-E_\ell )}=: U_\ell(X,u),
 \end{equation}
 where on the right-hand side one inserts the asymptotic expansions as $u\to 0-$, and on the left-hand side one defines $(1+\p_a)^{X}:=\exp(-X\sum_{j\geq 1} \frac {(-1)^{j}} j \p_a)$. Furthermore, the coefficient of each power of $u$ in the series on either side is a polynomial in $X$.
\end{prop}

Let us note that, from the standard recurrence relation $\Gamma(z+1)=z\Gamma(z)$, one obtains the formula 
 \begin{equation}\label{equ:Urecurrence}
     U_\ell(X+p,u)=(-\lambda_\ell)^p (-E_\ell+X)(-E_\ell+X+1)\cdots (-E_\ell+X+p-1) U_\ell(X,u),
 \end{equation}
for any non-negative integer $p$. 

The formula \eqref{equ:propUdef} is a variation of a formula found in \cite[\S 2.2]{TT2} and \cite[Proposition 15.7]{Turchin}. The only difference is that in loc. cit. one replaces $\lambda_\ell$ by $\ell u^\ell$. However, this change does not alter the proof of the formula. Hence we shall only sketch the derivation (due to Turchin) here. 

\begin{proof}[Proof of Proposition \ref{prop:Udef}]
First, one checks that both sides of \eqref{equ:propUdef} are power series in $u$ with coefficients that are polynomials in $X$. Since any polynomial is completely determined by its values on non-negative integers, it is sufficient to show \eqref{equ:propUdef} for $X=0,1,2,\dots$.
For $X=0$ the identity \eqref{equ:propUdef} is trivial -- both sides are $1$.
For $X=p$ a non-negative integer one has, using \eqref{equ:Urecurrence},
\[
U_\ell(p,u) = \lambda_\ell^p E_\ell(E_\ell-1)\cdots (E_\ell-p+1).
\]
Similarly, the left-hand side of \eqref{equ:propUdef} becomes
\begin{align*}
&(1+\p_a)^{p}e^{-a}(1+ \lambda_\ell a)^{E_\ell}\mid_{a=0}
=\p_a^p(1+ \lambda_\ell a)^{E_\ell}\mid_{a=0}
=\lambda_\ell^p E_\ell(E_\ell-1)\cdots (E_\ell-p+1),
\end{align*}
and the proposition follows.
\end{proof}
 
 Recall Stirling's asymptotic expansion of the $\Gamma$ function 
 \begin{equation}
 \label{equ:Gamma asymptotic} 
 \log\Gamma(z) \sim (z-\frac 1 2)\log z -z +\frac12 \log(2\pi) +B(z)
 \quad \text{as $z\to \infty$} \\
 \end{equation}
 with 
 \begin{equation} \label{equ:B def}
  B(z) = \sum_{r\geq 2}\frac{B_r}{r(r-1)} \frac 1 {z^{r-1}}.
\end{equation}
Using this expansion one can write $U_\ell(X,u)$ more explicitly as
 \[
 \log U_\ell(X,u)
 =   
 \log \frac {(-\lambda_\ell)^X \Gamma(-E_\ell+X) }{\Gamma(-E_\ell)}
     =
     X\left(\log(\lambda_\ell E_\ell)-1 \right)+(-E_\ell+X-\frac 1 2 )\log(1-\frac X{E_\ell}) + B(-E_\ell+X)- B(-E_\ell).
 \]

 Furthermore, note that we may compute, for non-negative integers $k$,
 \[
 (\log(1+\p_a))^k(1+\p_a)^{X}e^{-a}(1+\lambda_\ell a)^{E_\ell}\mid_{a=0}
 =
 \p_X^k U_\ell(X,u).
 \]
 
 For later use, let us list the following special cases: 
 \begin{multline}
     \label{equ:Ueqns_log}
      (\log(1+\p_a))(1+\p_a)^{X}e^{-a}(1+\lambda_\ell a)^{E_\ell}\mid_{a=0} 
     =\partial_X U_\ell(X,u)=
     \left( \log(\lambda_\ell(E_\ell-X)) + \psi_0(-E_\ell+X) \right) U_\ell(X,u)  
 \end{multline}
 \begin{multline}
     \label{equ:Ueqns_log1}
     (\log(1+\p_a))(1+\p_a)^{X+1}e^{-a}(1+\lambda_\ell a)^{E_\ell}\mid_{a=0} 
     =\partial_X U_\ell(X+1,u)
     \\=
     \left( \log(\lambda_\ell(E_\ell-X)) + \psi_0(-E_\ell+X) + \frac 1 {-E_\ell+X} \right) (-\lambda_\ell)(-E_\ell+X) U_\ell(X,u) 
 \end{multline}
 \begin{multline}
     \label{equ:Ueqns_loglog}
     (\log(1+\p_a))^2(1+\p_a)^{X}e^{-a}(1+\lambda_\ell a)^{E_\ell}\mid_{a=0} 
     =\partial_X^2 U(X,u)
     \\=
     \left( (\log(\lambda_\ell(E_\ell-X))+\psi_0(-E_\ell+X))^2 + \psi_1(-E_\ell+X)    \right) U_\ell(X,u) 
 \end{multline}
 \begin{multline}
     \label{equ:Ueqns_loglog1}
     (\log(1+\p_a))^2(1+\p_a)^{X+1}e^{-a}(1+\lambda_\ell a)^{E_\ell}\mid_{a=0} 
     =\partial_X^2 U_\ell(X+1,u)
     \\=
     \left( \left(\log(\lambda_\ell(E_\ell-X))+\psi_0(-E_\ell+X) + \frac 1 {-E_\ell+X}\right)^2 + \psi_1(-E_\ell+X) - \frac{1}{(-E_\ell+X)^2}    \right) (-\lambda_\ell)(-E_\ell+X) U_\ell(X,u)
 \end{multline}
 
 Here we use the recurrence relation \eqref{equ:Urecurrence}
 and the notation \eqref{equ:psi def} for the digamma and trigamma series.

\subsection{Euler characteristic of $\ftX$}
\label{sec:eulerchar proof}

The graded vector space $\fHHGC$ has additional gradings from the complexity-grading, and from the number of $\epsilon$ and $\omega$-legs.
We can hence define the trigraded $\bbS$-equivariant Euler characteristic 
\[
\chi^{u,v,w}(\fHHGC) \in \Lambda\llbracket u,v,w \rrbracket,
\]
with $u, v,w$ being the formal variables tracking the complexity and the number of $\epsilon$- and $\omega$-legs, respectively. In other words, the Euler characteristic of the subcomplex with $k$ $\epsilon$- and $l$ $\omega$-legs and complexity $m$ is the coefficient of $u^mv^kw^l$ in the formal power series.
Then we have:

\begin{prop}\label{prop:eulerchar ftX}
  The equivariant Euler characteristic of $\ftX$ is 
  \[
    \chi^{u,v,w}(\ftX) = \prod_\ell U_\ell\left(\frac 1 \ell \sum_{d\mid \ell} \mu(\ell/d) (p_d +v^d+w^d ), u \right)
  \]
  with the function $U_\ell$ defined in Proposition \ref{prop:Udef} above.  
\end{prop}

\begin{proof}

We note that
\[
\chi^{v,w}(\DS \boxtimes \Com_1\boxtimes \Com_1) 
= 
 \exp\left(\sum_{\ell\geq 1} \frac 1 \ell (q_\ell  +v^\ell +w^\ell ) p_\ell \right)
\in \Lambda_{p,q}\llbracket v,w \rrbracket.
\]
Hence 
\begin{align*}
\chi^{v,w}( (\DS \boxtimes \Com_1\boxtimes \Com_1)\circ \Lie_0)
&=
\exp\left( \sum_{\ell\geq 1} \sum_{k\geq 1}
\frac {\mu(k)} {\ell k} (q_\ell + v^\ell+w^\ell ) \log(1+p_{k\ell})\right)
\\&=
\exp\left( 
  \sum_{\ell\geq 1}
\frac 1 \ell \sum_{d\mid \ell} \mu(\ell/d) (q_d + v^d + w^d )\log(1+p_{\ell}) \right)
\\&=
\prod_{\ell\geq 1}
(1+p_{\ell})^{\frac 1 \ell \sum_{d\mid \ell} \mu(\ell/d) (q_d + v^d+ w^d )}.
\end{align*}

We hence compute, using Lemma \ref{lem:ftX} and Proposition \ref{prop:iHom Gr}:
\begin{align*}
  \chi^{u,v,w}(\fHHGC)
  &= \chi^{u,w}(( \DS \boxtimes \Com_1\boxtimes \Com_1) \otimes_{\bbS} \ftG ) 
  &
  \\&=
  \chi^{u,v,w}( (\DS \boxtimes \Com_1\boxtimes \Com_1)\circ \Lie_0)  \otimes_{\bbS} \cro\BV_0)
  &\text{by Proposition \ref{prop:iHom Gr}}
  \\&=
  \langle 
   \chi^{v,w}(\DS \boxtimes \Com_1\boxtimes \Com_1)\circ \chi(\Lie_0)), \chi^{u}(\croBV_0) 
  \rangle
  &\text{by \eqref{equ:tens prod}}
  \\&=
  \prod_\ell 
  (1+\partial_a)^{\frac 1 \ell \sum_{d\mid \ell} \mu(\ell/d) (q_d  +v^d+w^d )}
  e^{-a }(1+ u^\ell (1-u^\ell) \ell a)^{ \frac 1 \ell \sum_{d\mid \ell}\mu(\ell/d) \frac 1 {u^d} }
  \mid_{a=0}
  &\text{by \eqref{equ:prod def}}
  \\&=
  \prod_\ell U_\ell\left(\frac 1 \ell \sum_{d\mid \ell} \mu(\ell/d) (q_d +v^d+w^d ), u \right)
  .
\end{align*} 
In the last line we used the function $U_\ell$ of Proposition \ref{prop:Udef}.
Note that here we had to work with bisymmetric sequences and hence two sets of power sums $p_j$, $q_j$, so that in the final expression the formula is in $\Lambda_q=\Q\llbracket q_1,q_2,\dots \rrbracket$ instead of $\Lambda$. But the trivial replacement $q_d\to p_d$ yields the proposition. 
\end{proof}

We note that the part of $\ftX$ spanned by graphs without any $\epsilon$- or $\omega$-legs agrees with $\ftG$.
Hence we obtain, by setting $v=w=0$ in the formula of Proposition \ref{prop:eulerchar ftX}:
\begin{cor}
\[
  \chi^u(\ftG) = \chi^{u,v,w}(\fHHGC )(v=w=0)=
  \prod_\ell U_\ell\left(\frac 1 \ell \sum_{d\mid \ell} \mu(\ell/d) p_d , u \right)  
\]
\end{cor}
\noindent We note that this corollary has been found earlier in \cite{TT2}.

\subsection{Euler characteristic of the connected part with two $\omega$-legs}\label{sec:eulerchar fXconn2}
Next we reduce the computation of the Euler characteristic of $\tX_{g,n}$ to that of $\ftX_{g,n}$.
As a first step we defined the sub-symmetric sequence 
$
  \fHHGC^{conn} \subset \fHHGC
$
generated by graphs in which each connected component has at least one $\epsilon$- or $\omega$-leg.
We then have 
\[
  \fHHGC = \fHHGC^{conn} \boxtimes \ftG.
\]
Hence the Euler characteristic of $\fHHGC^{conn}$ is
\[
  \chi^{u,v,w}(\fHHGC^{conn} )
  =
  \frac{ \chi^{u,v,w}(\fHHGC ) }
  {\chi^{u,v,w}(\fHHGC )(v=w=0) }
  =
  \prod_\ell 
  \frac { 
    U_\ell(\frac 1 \ell \sum_{d\mid \ell} \mu(\ell/d) (p_d  +v^d+w^d ), u )
  }
  { 
    U_\ell(\frac 1 \ell \sum_{d\mid \ell} \mu(\ell/d) p_d, u )
  }\, .
\]

Furthermore, we shall only be interested in the
subcomplexes
\[
  \fHHGC^{conn,j} \subset \fHHGC^{conn}
\]
spanned by graphs with $j=0,1,2$ $\omega$-legs, and we do not want to fix the number of $\epsilon$-legs since the latter is not invariant under the differential.
The relevant Euler characteristic is computed by the  coefficient of $w^j$, evaluated at $v=1$. 
For example, for $j=0$ we set $w=0$ to obtain 
\begin{align*}
  \chi^u(\fHHGC^{conn,0} )
  &=\chi^{u,v,w}(\fHHGC^{conn} )(v=1, w=0) 
\\&=
    \prod_{\ell \geq 1}
  \frac { 
    U_\ell(\frac 1 \ell \sum_{d\mid \ell} \mu(\ell/d) (p_d  +1), u)
  }
  { 
    U_\ell(\frac 1 \ell \sum_{d\mid \ell} \mu(\ell/d) p_d , u)
  }
\\
&=
\frac { 
    U_1( p_1  +1, u )
  }
  { 
    U_1( p_1 , u )
  }
 \ \ = \ \ 
  u(1-u)Y_1^-,
\end{align*}
where we used that $\sum_{d\mid \ell}\mu(\ell/d)=\delta_{1\ell}$, equation \eqref{equ:Urecurrence}, and the abbreviation
\[
  Y_\ell^\pm := \sum_{d\mid\ell}\mu(\ell/d) (u^{-d} \pm p_d)).
\]
Next, for $j=1$:
\begin{align*}
  \chi^u(\fHHGC^{conn,1} )
  &= \partial_w\mid_{w=0} \chi^{u,v,w}(\fHHGC^{conn} )(v=1) 
  \\&=
  \frac { 
    \p_X U_1( p_1  +1,  u )
  }
  { 
    U_1( p_1 , u )
  }
  + 
  \frac { 
    U_1( p_1  +1,  u )
  }
  { 
    U_1( p_1 , u )
  }
  \left(\sum_{\ell\geq 2} \frac{\mu(\ell)}{\ell}
  \frac { 
     \partial_X U_\ell(\frac 1 \ell \sum_{d\mid \ell} \mu(\ell/d) p_d, u)
  }
  { 
    U_\ell(\frac 1 \ell \sum_{d\mid \ell} \mu(\ell/d) p_d , u )
  }\right) 
\\&=
u(1-u)Y_1^-
\left(-\frac 1 {Y_1} +\sum_{\ell\geq 1} \frac{\mu(\ell)}{\ell}
  \left(\log(\ell u^\ell(1-u^\ell) Y_\ell) 
  + 
  \psi_0(- Y_\ell^-) \right)
  \right)
\end{align*}
To obtain the second line, we used again that $\sum_{d\mid \ell}\mu(\ell/d)=\delta_{1\ell}$.
The first term in the second line is the evaluated using \eqref{equ:Ueqns_log1}, while the other derivatives are computed using \eqref{equ:Ueqns_log}.
Factoring out the common factor $u(1-u)Y_1^-$ then produces the third line.
Similarly, one computes the second derivative.
\begin{equation*} 
\resizebox{.95\hsize}{!}{
$\begin{aligned}
 &\chi^u(\fHHGC^{conn,2} ) = \frac12 
  \partial_w^2\mid_{w=0} \chi^{u,v,w}(\fHHGC^{conn} )(v=1) 
  \\
  &=
  \frac 12  \frac { 
    \p_X^2 U_1( p_1  +1,  u )
  }
  { 
    U_1( p_1 , u )
  }
  + 
  \frac 12 
  \frac { 
    U_1( p_1  +1,  u )
  }
  { 
    U_1( p_1 , u )
  }
  \left(\sum_{\ell\geq 2} \left(\frac{\mu(\ell)}{\ell}\right)^2
  \frac { 
     \partial_X^2 U_\ell(\frac 1 \ell \sum_{d\mid \ell} \mu(\ell/d) p_d, u)
  }
  { 
    U_\ell(\frac 1 \ell \sum_{d\mid \ell} \mu(\ell/d) p_d , u )
  }\right)
  \\ &\quad +
  \frac { 
    U_1( p_1  +1,  u )
  }
  { 
    U_1( p_1 , u )
  }
  \left(\sum_{\ell\geq 2\atop 2\mid\ell} \frac{\mu(\ell/2)}{\ell}
  \frac { 
     \partial_X U_\ell(\frac 1 \ell \sum_{d\mid \ell} \mu(\ell/d) p_d, u)
  }
  { 
    U_\ell(\frac 1 \ell \sum_{d\mid \ell} \mu(\ell/d) p_d , u )
  }\right)
  +
  \frac { 
    \p_X U_1( p_1  +1,  u )
  }
  { 
    U_1( p_1 , u )
  }
  \left(\sum_{\ell\geq 2} \frac{\mu(\ell)}{\ell}
  \frac { 
     \partial_X U_\ell(\frac 1 \ell \sum_{d\mid \ell} \mu(\ell/d) p_d, u)
  }
  { 
    U_\ell(\frac 1 \ell \sum_{d\mid \ell} \mu(\ell/d) p_d , u )
  }\right)  \\&\quad
 +
  \frac { 
    U_1( p_1  +1,  u )
  }
  { 
    U_1( p_1 , u )
  }
  \sum_{\ell>\ell'\geq 2}
  \left(\frac{\mu(\ell)}{\ell}
  \frac { 
     \partial_X U_\ell(\frac 1 \ell \sum_{d\mid \ell} \mu(\ell/d) p_d, u)
  }
  { 
    U_\ell(\frac 1 \ell \sum_{d\mid \ell} \mu(\ell/d) p_d , u )
  }\right) 
  \left(\frac{\mu(\ell')}{\ell'}
  \frac { 
     \partial_X U_{\ell'}(\frac 1 {\ell'} \sum_{d\mid \ell'} \mu(\ell'/d) p_d, u)
  }
  { 
    U_{\ell'}(\frac 1 {\ell'} \sum_{d\mid \ell'} \mu(\ell'/d) p_d , u )
  }\right) 
  \\
  &= \frac12 u(1-u)Y_1^-
  \left[
    \left(-\frac 1 {Y_1^-} +\sum_{\ell\geq 1} \frac{\mu(\ell)}{\ell}
    \left(\log(\ell u^\ell(1-u^\ell) Y_\ell^-) 
    + 
    \psi_0(- Y_\ell^-) \right) \right)^2
  + 
  \sum_{\ell\geq 1} \frac{\mu(\ell)}{\ell}
  \left(\log(2\ell u^{2\ell}(1-u^{2\ell}) Y_{2\ell}^-) 
  + 
  \psi_0(- Y_{2\ell}^-) \right) 
  \right.
\\&\quad\quad\quad
  \left.
     -\frac 1 {(Y_1^-)^2} + \sum_{\ell\geq 1} \frac{\mu(\ell)^2}{\ell^2}
    \psi_1(-Y_\ell^-)
  \right]
    \end{aligned}$
    }
\end{equation*}

\noindent To obtain the final equality we evaluated all terms using \eqref{equ:Urecurrence} and \eqref{equ:Ueqns_log}-\eqref{equ:Ueqns_loglog1}, and simplified the resulting expression. The steps of the final simplification are elementary and are omitted here, only the end result is shown.

\subsection{Euler characteristic of $\tX$}

Here, we show:
\begin{prop}\label{prop:eulerchar tX}
\begin{align*}
\chi^u(\tX)
= 
\chi^u( \fHHGC^{conn,2} )
+
u \chi^u( \fHHGC^{conn,1} )
+(u+u^2) \chi^u( \fHHGC^{conn,0} )
\end{align*}
\end{prop}
\begin{proof}
The symmetric sequence $\tX$ differs from $\fHHGC^{conn,2}$ only in so far that generators in $\tX$ are not allowed to contain $(\epsilon,\omega)$- or $(\omega,\omega)$-edges, while we have not imposed such a condition in $\fHHGC^{conn,2}$. To show the proposition we hence have to correct for those graphs.

A general graph in $\fHHGC^{conn,2}$ with $(\omega,\omega)$-edge has the form 
\begin{align*}
  \begin{tikzpicture}
    \node (v1) at (0,0) {$\omega$};
    \node (v2) at (0.7,0) {$\omega$};
    \draw (v1) edge (v2);
    \node[draw, circle, minimum width=1cm, minimum height=.5cm] (v) at (2,0) {$\Gamma_0$};
  \end{tikzpicture},
\end{align*}
with the right-hand part $\Gamma_0$ a general graph in $\fHHGC^{conn,0}$. 
Hence to account for those graphs we have to subtract from $\chi^u(\fHHGC^{conn,2})$ 
the expression $-u \chi^u( \fHHGC^{conn,0} )$, the factor $-u$ accounting for the degree and genus shift introduced by the $(\omega,\omega)$-edge.

Similarly, a graph in $\fHHGC^{conn,2}$ with $(\epsilon,\omega)$-edge has the form 
\begin{align*}
  \begin{tikzpicture}
    \node (v1) at (0,0) {$\epsilon$};
    \node (v2) at (0.7,0) {$\omega$};
    \draw (v1) edge (v2);
    \node[draw, circle, minimum width=1cm, minimum height=.5cm] (v) at (2,0) {$\Gamma_1$};
  \end{tikzpicture},
\end{align*}
with $\Gamma_1\in \fHHGC^{conn,1}$, except for the caveat that $\Gamma_1$ must not contain an $(\epsilon,\omega)$-edge itself. But the graphs in $\fHHGC^{conn,1}$ with $(\epsilon,\omega)$-edge themselves have the form
\begin{align*}
  \begin{tikzpicture}
    \node (v1) at (0,0) {$\epsilon$};
    \node (v2) at (0.7,0) {$\omega$};
    \draw (v1) edge (v2);
    \node[draw, circle, minimum width=1cm, minimum height=.5cm] (v) at (2,0) {$\Gamma_0$};
  \end{tikzpicture},
\end{align*}
with $\Gamma_0$ again a general graph in $\fHHGC^{conn,0}$.
Adding and subtracting the corresponding Euler characteristics one has shown the Proposition.
\end{proof}

\subsection{Proof of Theorem \ref{thm:eulerchar}}
According to \eqref{equ:tXX} the complexes $\tX_{g,n}$ are obtained from $X_{g,n}$ by a degree shift and multiplication by the sign representation $\sgn_n$. Hence we find that
\[
\chi^u(X) = - \chi^u(\tX)(-q_n \leftarrow q_n),  
\]
with the notation meaning that each occurrence of $q_n$ should be replaced by $-q_n$ on the right-hand side.
But using Proposition \ref{prop:eulerchar tX} and the formulas of \S \ref{sec:eulerchar fXconn2}, we hence see that
\begin{equation*} 
\resizebox{.95\hsize}{!}{
$\begin{aligned}
  \chi^u(X)
    &= 
    \frac { u(u-1) Y_1^+} 2\left[
    \sum_{\ell\geq 1} \frac{\mu(\ell)}\ell
    (\log(-2\ell u^{2\ell}(1-u^{2\ell})Y_{2\ell}^+) +\psi_0(-Y_{2\ell}^+))
    +
    \left(-\frac 1 {Y_1^+} + \sum_{\ell\geq 1}\frac{\mu(\ell)}\ell(\log(-\ell u^\ell(1-u^\ell)Y_\ell^+) +\psi_0(-Y_\ell^+)) \right)^2
    \right.\\&\quad \left.
    -\frac 1 {(Y_1^+)^2}
    +
    \sum_{\ell\geq 1}\frac{\mu(\ell)^2}{\ell^2}
    \psi_1(-Y_\ell^+)
    +2u \left(-\frac 1 {Y_1^+}+\sum_{\ell\geq 1}\frac{\mu(\ell)}\ell(\log(-\ell u^\ell(1-u^\ell)Y_{\ell}^+) +\psi_0(-Y_\ell^+)) \right)
    \right]
    + u^2 (u^2-1) Y_1^+.
    \end{aligned}$
    }
\end{equation*}

This expression can be simplified further.
First, one has the following equality of formal power series
\[
\sum_{\ell \geq 1} \frac{\mu(\ell)}{\ell} \log(1-x^\ell)
=
-\sum_{\ell,n \geq 1} \frac{\mu(\ell)}{\ell} \frac {1}{n} x^{n\ell}
=
-\sum_{N\geq 1}\left(\sum_{d\mid N} \mu(d)\right) x^{N} = -x.
\]
Using this formula three times to absorb the factors $(1-u^{\ell})$ and $(1-u^{2\ell})$ inside the logarithms we obtain 
\begin{equation*} 
\resizebox{.95\hsize}{!}{
$\begin{aligned}
  \chi^u(X)
  &= 
  \frac { u(u-1) Y_1^+} 2\left[
  \sum_{\ell\geq 1} \frac{\mu(\ell)}\ell
  (\log(-2\ell u^{2\ell}Y_{2\ell}^+) +\psi_0(-Y_{2\ell}^+))
  +
  \left(-\frac 1 {Y_1^+} + \sum_{\ell\geq 1}\frac{\mu(\ell)}\ell(\log(-\ell u^\ell Y_\ell^+) +\psi_0(-Y_\ell^+)) \right)^2
  \right.\\&\quad \left.
    -u^2 + u^2-2u^2
    -\frac 1 {(Y_1^+)^2}
 + 
  \sum_{\ell\geq 1}\frac{\mu(\ell)^2}{\ell^2}
  \psi_1(-Y_\ell^+)
  \right]
  + u^2 (u^2-1) Y_1^+
  \\&= 
  \frac { u(u-1) Y_1^+} 2\left[
  \sum_{\ell\geq 1} \frac{\mu(\ell)}\ell
  (\log(-2\ell u^{2\ell}Y_{2\ell}^+) +\psi_0(-Y_{2\ell}^+))
  +
  \left(-\frac 1 {Y_1^+} + \sum_{\ell\geq 1}\frac{\mu(\ell)}l(\log(-\ell u^\ell Y_\ell^+) +\psi_0(-Y_\ell^+)) \right)^2
  \right.\\&\quad \left.
    -\frac 1 {(Y_1^+)^2}
 + 
  \sum_{\ell\geq 1}\frac{\mu(\ell)^2}{\ell^2}
  \psi_1(-Y_\ell^+)
  \right]
  + u^2 (u-1) Y_1^+
  \, .
    \end{aligned}$
    }
\end{equation*}

\begin{proof}[Proof of Theorem \ref{thm:eulerchar}]
  To show Theorem \ref{thm:eulerchar} we just use the above formula for $\chi^u(X)$, with the following modifications.
  First, to obtain the generating function in terms of genera instead of complexity one has to perform the replacements 
  \begin{align}\label{equ:to replace}
    u&\to \hbar \text{ and }\ p_\ell\to \hbar^{-\ell}p_\ell, &&\text{or equivalently,} &  u&\to \hbar \text{ and } Y_\ell^+\to Z_\ell.
  \end{align}

  Second, one needs to mind that the sum \eqref{equ:omega2def} runs only over stable indices $2g+n\geq 3$, while $X_{0,2}\cong \Q[-1]$ is nontrivial by our definition. Hence we need to subtract a term 
  \begin{equation}\label{equ:subtract 1}
    -\frac 1 2 (p_1^2+p_2)=-\frac 1 2 (P_1^2+P_2)-P_1
  \end{equation}
  from the Euler characteristic to account for this difference.
  
  Finally, by Corollary \ref{cor:omega2 and X} there is a further correction $\hbar p_1$ arising from the contribution of $g=n=1$.
  Thus, we start from the formula for $\chi^u(X)$ above, then apply the substitution rule \eqref{equ:to replace}, add $\hbar p_1$ and subtract \eqref{equ:subtract 1} to finally obtain the formula of Theorem \ref{thm:eulerchar}. 
\end{proof}

\section{Discussion of terms in the Euler characteristic formula}
\label{sec:terms discussion}
The literature contains two complementary toolsets for computing dimensions (Hilbert series) and Euler characteristics of graph complexes. The first is the calculus of symmetric functions, typically paired with operadic methods, as in \cite{GK,TT2} and this paper.  The second is combinatorial counting weighted by automorphisms, along the lines of the P\'olya enumeration theorem, as in \cite{CFGP, WZ}.  When the first toolset is applicable, it typically yields relatively economical proofs and closed expressions for the generating functions. However, it also tends to obfuscate natural correspondences between terms in the resulting formulas and subsets of generators for the graph complex. In this section, we shall hence briefly discuss how the individual terms in the formulas of Theorem~\ref{thm:eulerchar} and Corollary~\ref{cor:eulerchar} relate to graph generators for $X_{g,n}$.

First recall from \cite[\S 6.1]{PayneWillwacher} that $X_{g,n}$ is quasi-isomorphic to its subcomplex $X^\star_{g,n}$, generated by graphs $\Gamma$ that are disjoint unions of one or two connected components with $\omega$-decorations together with one of the following: 
\[
(\text{empty graph}),
\quad 
\begin{tikzpicture}
  \node (v) at (0,0) {$\epsilon$};
  \node (w) at (0.7,0) {$\epsilon$};
  \draw (v) edge (w);
\end{tikzpicture},
\quad 
\begin{tikzpicture}
  \node (v) at (0,0) {$\epsilon$};
  \node (w) at (0.7,0) {$j$};
  \draw (v) edge (w);
\end{tikzpicture},
\quad
\begin{tikzpicture}
  \node (v) at (0,.4) {$\epsilon$};
  \node (w) at (0.7,.4) {$j$};
  \draw (v) edge (w);
  \node (vv) at (0,-.4) {$\epsilon$};
  \node (ww) at (0.7,-.4) {$\epsilon$};
  \draw (vv) edge (ww);
\end{tikzpicture}\ .
\]
The contribution of these four graphs to the generating function of the Euler characteristic is precisely 
\[
(1-\hbar)P_1,  
\]
which explains the corresponding factor in \eqref{equ:eulerchar}.

Now, consider the disjoint union of the connected components that contain $\omega$-decorations. 
In genus 0, the only possible graphs are pairs of two trees, as shown.
\[
  \begin{tikzpicture}[yscale=-1]
    \node[int] (v1) at (0,.8) {};
    \node[] (w1) at (0,1.5) {$\omega$};
    \node[int] (v2) at (-.5,.1) {};
    \node[] (w2) at (-1,-.4) {$1$};
    \node[] (e) at (0,-.4) {$3$};
    \node[] (w3) at (.5,.1) {$2$};
    \draw (v1) edge (v2) edge (w3) edge (w1)
    (v2) edge (e) edge (w2);
    \begin{scope}[xshift=2cm]
      \node[int] (v1) at (0,.8) {};
    \node[] (w1) at (0,1.5) {$\omega$};
    \node[] (v2) at (-.5,.1) {$4$};
    \node[] (w3) at (.5,.1) {$5$};
    \draw (v1) edge (w3) edge (w1)
    (v1) edge (v2);
    \end{scope}
  \end{tikzpicture}
\]
The Euler characteristic of the graph complex of rooted trees, like that of $\Lie_0$, is equal to $\sum_{\ell\geq 1} \frac{\mu(\ell)}{\ell} \log P_\ell$. 
Since we have two trees in the graph, we have to take a symmetric product of two such rooted tree complexes.
On the level of Euler characteristics taking the symmetric product translates to the plethysm with the symmetric function $h_2=\frac12(p_1^2+p_2)$, and 
\[
h_2\circ \left(\sum_{\ell\geq 1} \frac{\mu(\ell)}{\ell} \log P_\ell  \right)
=
\frac 1 2 \left[ \left( \sum_{\ell\geq 1} \frac{\mu(\ell)}{\ell} \log P_\ell \right)^2
  + \sum_{\ell\geq 1} \frac{\mu(\ell)}{\ell} \log P_{2\ell} \right],
\]
explaining the corresponding terms in the genus 0 and 1 Euler characteristics.

In genus $g \geq 2$, there are still some graphs with one of the two $\omega$-legs being the root of a tree.
The corresponding contributions to the Euler characteristic hence have a factor $\sum_{\ell\geq 1} \frac{\mu(\ell)}{\ell} \log P_\ell$ as well from the tree part, which is multiplied by the contribution from the non-tree part.

Let us next turn to the remaining terms in the Euler characteristic formulas, coming from the non-tree connected components in graphs. 
The characteristic feature of those terms is that they are finite linear combinations of monomials in the $P_j$ and $P_j^{-1}$. 
This can be seen directly from using graph counting techniques to compute the Euler characteristic, using the same strategy of proof as in \cite{CFGP}.

Consider an arbitrary graph $\Gamma$ in $X_{g,n}$. We call the \emph{core} $\gamma=[\Gamma]$ of $\Gamma$ the graph obtained by the following algorithm:
\begin{itemize}
  \item Remove all numbered external vertices and their adjacent edges.
  \item Recursively remove all univalent internal vertices thus created, with their adjacent edges.
  \item Remove the bivalent vertices thus produced and merge the two edges adjacent to them.
\end{itemize}

Let $X_{g,n}^\gamma\subset X_{g,n}$ be the subcomplex spanned by graphs with core $\gamma$, and set $$z_\gamma := \sum_{n\geq 0} \chi^\bbS(X_{g,n}^\gamma).$$ Then we have
$
\sum_{n\geq 0} \chi^\bbS(X_{g,n})  =
\sum_{\gamma} z_\gamma
$
with the first sum being over all (isomorphism classes of) genus $g$ cores $\gamma$. Assume for simplicity that the core $\gamma$ only has connected components of genera $\geq 1$. 
Then one may show as in \cite[Proposition 3.2]{CFGP} that 
\begin{equation}\label{eq:monomials}
  z_\gamma = \frac1{|\Aut (\gamma)|}
  =
  \sum_{\tau\in \Aut (\gamma)}
  \frac{P^{i(\tau_V)}P^{i(\tau_E)}}{P^{i(\tau_H)}}
\end{equation}
with $\tau_V$, $\tau_E$, $\tau_H$ the permutations on the sets of vertices, edges and half-edges of $\gamma$, with $i(\sigma)=(i_1(\sigma),\dots)$ the cycle type of a permutation $\sigma$, and with the multi-index notation 
\[
  P^{i(\sigma)} = P_1^{i_1(\sigma)} P_{2}^{i_2(\sigma)}\cdots \in \Lambda. 
\]
Since in every genus there are only finitely many possible core graphs contributing, the Euler characteristic in this genus must hence be a finite linear combination of such Laurent monomials in the $P_i$, of the form given by \eqref{eq:monomials}.

Next suppose that the core consists of a single isolated external vertex decorated by $\omega$, and another higher genus component. The graphs with this core are unions of trees with $\omega$-labelled root and some other graph with only one $\omega$-decoration. 
\[
  \begin{tikzpicture}[yscale=-1]
    \node[int] (v1) at (0,.8) {};
    \node[] (w1) at (0,1.5) {$\omega$};
    \node[int] (v2) at (-.5,.1) {};
    \node[] (w2) at (-1,-.4) {$1$};
    \node[] (e) at (0,-.4) {$3$};
    \node[] (w3) at (.5,.1) {$2$};
    \draw (v1) edge (v2) edge (w3) edge (w1)
    (v2) edge (e) edge (w2);
    \begin{scope}[xshift=2cm]
      \node[int] (v1) at (0,.8) {};
    \node[] (w1) at (0,1.5) {$\omega$};
    \node[int] (w4) at (0,-.6) {};
    \node[int] (v2) at (-.5,.1) {};
    \node[int] (w3) at (.5,.1) {};
    \node[] (w3b) at (1.2,.1) {$4$};
    \draw (v1) edge (w3) edge (w1)
    (v1) edge (v2) edge (w4)
    (w4) edge (w3) edge (v2)
    (w3) edge (v2) edge (w3b);
    \end{scope}
  \end{tikzpicture}
\]
These graphs then contribute summands of the form 
\[
  \left(\sum_{\ell\geq 1} \frac{\mu(\ell)}{\ell} \log P_\ell\right) 
  f(P_1,P_2,\dots),
\]
with $f$ some finite sum of Laurent monomials of the special form given by \eqref{eq:monomials}.  
This then explains the structure of the terms appearing in Corollary \ref{cor:eulerchar}.

\end{document}